\documentclass[a4paper,11pt,leqno,english]{amsart}
\usepackage{enumerate}
\usepackage{amssymb,latexsym}
\usepackage[T1]{fontenc}
\usepackage[centering]{geometry}
\usepackage{url}
\usepackage[french, main=english]{babel}
\usepackage[utf8]{inputenc}
\usepackage[all]{xy}
\usepackage{mathrsfs}
\usepackage[dvipsnames]{xcolor}
\usepackage{comment}
\definecolor{violet}{rgb}{0.0,0.2,0.7}
\definecolor{rouge2}{rgb}{0.8,0.0,0.2}
\usepackage{tikz}
\usepackage{empheq}
\usepackage{tikz-cd}
\usetikzlibrary{matrix,arrows,decorations.pathmorphing}
\usepackage[backref=page]{hyperref}
\usepackage[nameinlink, capitalize]{cleveref}
\hypersetup{
    bookmarks=true,         
    unicode=false,          
    pdftoolbar=true,        
    pdfmenubar=true,        
    pdffitwindow=false,     
    pdfstartview={FitH},    
    pdftitle={},            
    pdfauthor={},           
    colorlinks=true,        
    linkcolor=rouge2,       
    citecolor=violet,       
    filecolor=black,        
    urlcolor=cyan}          
\usepackage{enumitem, xspace}
\usepackage{appendix}

\setlist[enumerate]{
  label=(\thethm.\arabic*),
  before={\setcounter{enumi}{\value{equation}}},
  after={\setcounter{equation}{\value{enumi}}},
  itemsep=1ex
}

\setlist[itemize]{
  leftmargin=*,
  topsep=1ex,
  itemsep=1ex,
  label=$\circ$
}

\theoremstyle{plain}    
\newtheorem{thm}{Theorem}[section]
\theoremstyle{plain} 
\newtheorem{bigthm}{Theorem}

\renewcommand{\thebigthm}{\Alph{bigthm}} 
 \numberwithin{equation}{thm} 
 \numberwithin{figure}{section} 
 \newtheorem{cor}[thm]{Corollary} 
 \theoremstyle{plain}    
 \newtheorem{prop}[thm]{Proposition} 
 \theoremstyle{plain}    
 \newtheorem{lem}[thm]{Lemma} 
 \theoremstyle{remark}
 \newtheorem{claim}[thm]{Claim} 
 \theoremstyle{remark}
 \newtheorem{rem}[thm]{Remark}
 \newtheorem*{rem-plain}{Remark}
\newtheorem{exa}[thm]{Example}
\theoremstyle{definition}  
\newtheorem{setup}[thm]{Setup}
\theoremstyle{plain}

\theoremstyle{definition}
\newtheorem{defi}[thm]{Definition}

\newtheorem{conj}[thm]{Conjecture}

\newcommand{\inv}{^{-1}}
\newcommand{\from}{\colon}

\newcommand{\imp}{\Rightarrow}
\newcommand{\lto}{\longrightarrow}

\newcommand{\x}{\times}
\newcommand{\inj}{\hookrightarrow}

\newcommand{\bij}{\xrightarrow{\,\smash{\raisebox{-.5ex}{\ensuremath{\scriptstyle\sim}}}\,}}
\newcommand{\isom}{\cong}
\newcommand{\defn}{\coloneqq}

\newcommand{\tensor}{\otimes}

\newcommand{\wt}{\widetilde}

\newcommand{\wh}{\widehat}

\newcommand{\dd}{\mathrm d}

\newcommand{\dual}{^{\smash{\scalebox{.7}[1.4]{\rotatebox{90}{\textup\guilsinglleft}}}}}
\newcommand{\ddual}{^{\smash{\scalebox{.7}[1.4]{\rotatebox{90}{\textup\guilsinglleft} \hspace{-.5em} \rotatebox{90}{\textup\guilsinglleft}}}}}
\newcommand{\acts}{\ \rotatebox[origin=c]{-90}{\ensuremath{\circlearrowleft}}\ }

\newcommand{\factor}[2]{\left. \raise 2pt\hbox{$#1$} \right/\hskip -2pt \raise -2pt\hbox{$#2$}}

\newdir{ ir}{{}*!/-5pt/@^{(}} 
\newdir{ il}{{}*!/-5pt/@_{(}} 

\DeclareMathOperator{\rk}{rk}

\DeclareMathOperator{\tor}{tor}

\DeclareMathOperator{\Exc}{Exc}

\DeclareMathOperator{\supp}{supp}

\DeclareMathOperator{\Deflt}{Def^{lt}}

\DeclareMathOperator{\Gr}{Gr}

\newcommand{\set}[1]{\left\{ #1 \right\}}

\def\rd#1.{\lfloor{#1}\rfloor}
\def\rp#1.{\lceil{#1}\rceil}
\def\tw#1.{\langle{#1}\rangle}

\newcommand{\la}{\langle}
\newcommand{\ra}{\rangle}

\renewcommand{\O}[1]{\mathscr{O}_{#1}}
\newcommand{\Omegap}[2]{\Omega_{#1}^{#2}}
\newcommand{\Omegar}[2]{\Omega_{#1}^{[#2]}}
\newcommand{\T}[1]{\mathscr{T}_{#1}}
\newcommand{\can}[1]{\omega_{#1}}

\newcommand{\Reg}[1]{{#1}_{\mathrm{reg}}}
\newcommand{\sing}[1]{{#1}_{\mathrm{sg}}}

\newcommand{\codim}[2]{\mathrm{codim}_{#1}(#2)}

\newcommand{\cc}[2]{\mathrm{c}_{#1}(#2)}
\newcommand{\cpc}[3]{\mathrm{c}_{#1}^{#2}(#3)}
\newcommand{\ccorb}[2]{\mathrm{\tilde c}_{#1}(#2)}

\newcommand{\PH}[1]{\mathrm{PH}_{#1}}

\def\Hnought#1.#2.{\mathit{\Gamma} \!\left( #1, #2 \right)}
\def\HH#1.#2.#3.{\mathrm{H}^{#1} \!\left( #2, #3 \right)}
\def\HHs#1.#2.{\mathrm{H}^{#1} \!\left( #2 \right)}
\def\hh#1.#2.#3.{h^{#1} \!\left( #2, #3 \right)}
\def\RR#1.#2.#3.{R^{#1} #2_* #3}
\def\HHc#1.#2.#3.{\mathrm{H}_{\mathrm{c}}^{#1} \!\left( #2, #3 \right)}
\def\Hh#1.#2.#3.{\mathrm{H}_{#1} \!\left( #2, #3 \right)}
\def\Hhs#1.#2.{\mathrm{H}_{#1} \!\left( #2 \right)}
\def\Hom#1.#2.{\mathrm{Hom} \!\left( #1, #2 \right)}
\def\sHom#1.#2.{\mathscr{H}\!om \!\left( #1, #2 \right)}
\def\Ext#1.#2.#3.{\mathrm{Ext}^{#1} \!\left( #2, #3 \right)}
\def\sExt#1.#2.#3.{\mathscr{E}\!xt^{#1} \!\left( #2, #3 \right)}
\def\Link#1.#2.{\mathrm{Link} \!\left( #1, #2 \right)}

\newcommand{\C}{\ensuremath{\mathbb{C}}}
\newcommand{\N}{\ensuremath{\mathbb{N}}}
\newcommand{\PP}{\ensuremath{\mathbb{P}}}
\newcommand{\Q}{\ensuremath{\mathbb{Q}}}
\newcommand{\R}{\ensuremath{\mathbb{R}}}

\newcommand{\Z}{\ensuremath{\mathbb{Z}}}

\newcommand{\kahler}{K{\"{a}}hler\xspace}

\newcommand{\lt}{locally trivial\xspace}

\newcommand{\qe}{quasi-\'etale\xspace}

\def\1{\bold{1}}

\newcommand{\fA}{{\mathfrak{A}}}

\newcommand{\fX}{{\mathfrak{X}}}

\newcommand{\fY}{{\mathfrak{Y}}}

\newcommand{\wX}{\widehat{X}}
\newcommand{\wE}{\widehat{\mathscr E}}
\newcommand{\wF}{\widehat{\mathscr F}}

\newcommand{\om}{\omega}

\newcommand{\ome}{\om_{\ep}}

\newcommand{\ep}{\varepsilon}

\newcommand{\sE}{\mathscr E}
\newcommand{\sF}{\mathscr F}
\newcommand{\sG}{\mathscr G}

\newcommand{\sL}{\mathscr L}

\newcommand{\lref}{\labelcref}

\def\todo#1.#2{ %
  \textcolor{Mahogany}{ %
    \footnotesize %
    \newline %
    {\color{Mahogany}\fbox{\parbox{.97\textwidth}{\textbf{#1:} #2}}} %
    \newline %
  } %
}

\definecolor{forrest}{RGB}{81,133,49}
\definecolor{mydarkblue}{RGB}{10,92,153}

\title{Numerical characterization of complex torus quotients}

\author[Claudon]{Beno\^it Claudon}
\address{Univ Rennes, CNRS, IRMAR --- UMR 6625, F--35000 Rennes, France et Institut Universitaire de France}
\email{\href{mailto:benoit.claudon@univ-rennes1.fr}{benoit.claudon@univ-rennes1.fr}}
\urladdr{\href{https://perso.univ-rennes1.fr/benoit.claudon/}{perso.univ-rennes1.fr/benoit.claudon/}}

\author[Graf]{Patrick Graf}
\address{Lehrstuhl f\"ur Mathematik I, Universit\"at Bayreuth, 95440 Bayreuth, Germany}
\email{\href{mailto:patrick.graf@uni-bayreuth.de}{patrick.graf@uni-bayreuth.de}}
\urladdr{\href{http://www.pgraf.uni-bayreuth.de/en/}{www.graficland.uni-bayreuth.de}}

\author[Guenancia]{Henri Guenancia}
\address{Institut de Math\'ematiques de Toulouse, Universit\'e Paul Sabatier, 31062 Toulouse Cedex~9, France}
\email{\href{mailto:henri.guenancia@math.cnrs.fr}{henri.guenancia@math.cnrs.fr}}
\urladdr{\href{https://hguenancia.perso.math.cnrs.fr/}{hguenancia.perso.math.cnrs.fr/}}

\date{\today}

\hypersetup{
  pdfauthor={Beno\^it Claudon, Patrick Graf, and Henri Guenancia},
  pdftitle={Numerical characterization of complex torus quotients},
  pdfkeywords={Complex tori, vanishing Chern classes, klt singularities, stable sheaves, Bogomolov--Gieseker inequality},
  pdfstartview={Fit},
  pdfpagelayout={OneColumn},
  pdfpagemode={UseNone},
  linktoc=all,
  breaklinks,
  linkcolor=rouge2,
  citecolor=violet,
  urlcolor=[RGB]{0 96 0},
  colorlinks}

\begin{document}

\begin{abstract}
This article gives a characterization of quotients of complex tori by finite groups acting freely in codimension two in terms of a numerical vanishing condition on the first and second Chern class.
This generalizes results previously obtained by Greb--Kebekus--Peternell in the projective setting, and by Kirschner and the second author in dimension three.
As a key ingredient to the proof, we obtain a version of the Bogomolov--Gieseker inequality for stable sheaves on singular spaces, including a discussion of the case of equality.
\end{abstract}

\maketitle


\section{Introduction}

Let $X$ be a compact \kahler manifold of dimension $n$ such that $\cc1X=0\in \HH2.X.\R.$.
An important application of Yau's resolution of the Calabi conjecture \cite{Yau78} is the following uniformization result:
$X$ is isomorphic to the quotient $\factor TG$ of a complex torus $T$ by a finite group $G$ acting freely on $T$ if and only if there exists a \kahler class $\alpha \in \HH2.X.\R.$ such that $\cc2X\cdot \alpha^{n-2}=0$.

In recent years, a lot of effort has been devoted to generalizing the above uniformization criterion to a class of singular varieties arising naturally in the Minimal Model Program, called varieties with Kawamata log terminal singularities (klt for short). A first roadblock consists in defining Chern classes, e.g.~$\cc2X$.
This is quite delicate for singular varieties, as several possible definitions exist that do not coincide in general.
We will gloss over this problem in the introduction and refer to \cref{section chern classes} and the references therein for a more in-depth discussion.

In the projective case, i.e.~when $X$ is a projective variety with klt singularities with $\cc1X = 0$ and $\alpha = \cc1\sL$ is the class of an ample line bundle $\sL$, the uniformization problem has been solved by~\cite{GKP16} and~\cite{LuTaji18}.
More precisely, Greb--Kebekus--Peternell~\cite{GKP16} proved that if $X$ is smooth in codimension two, then $X$ is the quotient of an abelian variety provided that $\cc2X \cdot \alpha^{n-2} = 0$.
Later on, Lu--Taji~\cite{LuTaji18} were able to lift the assumption about the codimension of the singular locus using the theory of orbifold Chern classes initiated by Mumford~\cite{MR717614}.

In the transcendental case however, the slicing arguments used in \emph{loc.~cit.} to reduce to a complete intersection surface  are certainly not available anymore, and new ideas are required.
The threefold case has recently been settled by Kirschner and the second author~\cite{GK20}, but the key techniques therein do not seem to generalize to the higher dimensional case.
In this paper, we rely on the recent Beauville--Bogomolov decomposition theorem~\cite{BakkerGuenanciaLehn20} as well as our Bogomolov--Gieseker inequality, \cref{BG ineq} below, to settle the uniformization problem in the transcendental case, assuming that $X$ is smooth in codimension two.

\begin{bigthm} \label{thmA}
Let $X$ be a compact complex space of dimension $n$ with klt singularities, smooth in codimension two.
The following are equivalent:
\begin{enumerate}[label=(\thebigthm.\arabic*)]
\item\label{A.1} We have $\cc1X = 0 \in \HH2.X.\R.$, and there exists a \kahler class $\alpha \in \HH2.X.\R.$ as well as a resolution of singularities $f \from Y \to X$ which is an isomorphism over~$\Reg X$ such that
\[ \int_Y \cc2Y \wedge (f^* \alpha)^{n-2} = 0. \]
\item\label{A.2} There exists a complex $n$-torus $T$ and a holomorphic action of a finite group $G \acts T$, free in codimension two, such that $X \isom \factor TG$.
\end{enumerate}
\end{bigthm}

\subsection*{Strategy of proof of \cref{thmA}}

In the smooth case, the cohomological assumption $\cc1X = 0$ can be turned into the existence of a Ricci-flat \kahler metric $\omega$ in the class~$\alpha$~\cite{Yau78}.
The numerical condition $\cc2X \cdot \alpha^{n-2} = 0$ can then be easily translated into the vanishing of the full curvature tensor of $\omega$~\cite{CY75}.
When $X$ is merely a compact klt \kahler space with vanishing first Chern class, it is known that $X$ admits a singular Ricci-flat metric~\cite{EGZ}, but the curvature tensor of $(\Reg X, \omega)$ is not known to be $L^2$-integrable in general, and there does not seem to be a way to compute the quantity $\cc2X \cdot \alpha^{n-2}$ using $\omega$.

To circumvent this intrinsic difficulty, in addition to the Beauville--Bogomolov Decomposition Theorem cited above, we use the following statement that generalizes the classical Bogomolov--Gieseker inequality to the singular case.
We refer to \cref{section chern classes} for a more detailed discussion of the notion of Chern classes used below.

\begin{bigthm}[Bogomolov--Gieseker inequality] \label{BG ineq}
Let $X$ be a normal compact \kahler space of dimension $n$ together with a \kahler class $\alpha \in \HH2.X.\R.$.
Assume that $X$ is smooth in codimension two.
Furthermore, let $\sE$ be a rank~$r$ reflexive coherent sheaf on $X$ which is slope stable with respect to $\alpha$.
\begin{enumerate}[label=(\thebigthm.\arabic*)]
\setcounter{enumi}{0}
\item\label{B.1} The discriminant $\Delta(\sE) \defn 2r \, \cc2\sE - (r - 1) \, \cpc12\sE$ satisfies the inequality
\[ \Delta(\sE) \cdot \alpha^{n-2} \geq 0. \]
\item\label{B.2} If equality holds in~\lref{B.1}, then we have
\[ \Delta(\sE) \cdot \beta^{n-2} = 0 \]
for \emph{any} \kahler class $\beta \in \HH2.X.\R.$.
\item\label{B.3} If $\cc2\sE \cdot \alpha^{n-2} = \cpc12\sE \cdot \alpha^{n-2} = 0$, then we have
\[ \cc2\sE \cdot \beta^{n-2} = \cpc12\sE \cdot \beta^{n-2} = 0 \]
for any \kahler class $\beta \in \HH2.X.\R.$.
\end{enumerate}
\end{bigthm}

Let us give a couple of remarks about the result above. 
\begin{rem-plain}
Bogomolov--Gieseker inequality \ref{B.1} is also showed to hold for merely {\it semistable} sheaves, cf \cref{BG ss}. 
\end{rem-plain}

\begin{rem-plain}
The definition of the Chern classes used in \cref{BG ineq} does not rely on the assumption that $X$ is smooth in codimension two, and the statements continue to hold (with the same proof) if that assumption is dropped.
However, if $X$ has singularities in codimension two, these notions can behave quite counterintuitively.
For example, it may happen that $\cc1\sE$ is zero, but $\cpc12\sE$ is not (cf.~\cref{exa c2}).
For this reason, we have chosen to include the smoothness in codimension two assumption in the above statement.
\end{rem-plain}

With \cref{BG ineq} at hand, the roadmap to proving \cref{thmA} is then the following:
take $X$ as in the statement, i.e.~klt, smooth in codimension two and with $\cc1X = 0$.

\begin{itemize}
\item Deduce from~\lref{B.1} the semipositivity of the second Chern class, i.e.~$\cc2X \cdot \alpha^{n-2} \geq 0$ for any \kahler class $\alpha$.
This is not entirely immediate because the tangent sheaf $\T X$ will in general not be stable. Therefore, we first generalize~\lref{B.1} to semistable sheaves in \cref{BG ss}, from which the sought semipositivity follows easily.  
\item Decompose a cover of $X$ as a product of a complex torus, irreducible Calabi--Yau (ICY) and irreducible holomorphic symplectic (IHS) varieties.
Use the first step to show that if $\cc2X$ vanishes against a \kahler class, then the same is true of each factor in the above decomposition.
We are thus reduced to showing that for ICY and IHS varieties, $\mathrm c_2$ is \emph{strictly} positive against any \kahler class.
\item If $X$ is an ICY variety (hence projective), argue by contradiction.
Assuming that $\cc2X$ is zero against some \kahler class, use~\lref{B.3} to obtain vanishing against an \emph{ample} class.
This contradicts the projective case of \cref{thmA}.
\item If $X$ is an IHS variety, the required positivity is achieved via a complete description of $\cc2X$ using the (generalized) Fujiki relations, cf.~\cref{Fujiki}.
\end{itemize}

\subsection*{Relation to earlier and recent work}

As mentioned above, the projective case of \cref{thmA} was settled by Shepherd-Barron and Wilson~\cite{SBW94} in dimension three, and by Greb--Kebekus--Peternell~\cite{GKP16} in general, always assuming that $\alpha$ is a rational class.
Lu--Taji~\cite{LuTaji18} later removed the smoothness in codimension two assumption.

Inequality~\lref{B.1} has recently been obtained by Wu~\cite{Wu21}, but he does not discuss the case of equality.
Note that his result is formulated for polystable sheaves, but the proof actually assumes stability.
His line of argument is very similar to ours: pull everything back to a resolution and prove an openness of stability type statement there.
This basic idea goes back at least to~\cite[Proposition~6.9]{DP03}.

Even more recently, Chen and Wentworth~\cite{CW21} have likewise obtained a Bogomolov--Gieseker inequality in a setting similar to ours.
Their results, once combined with our previous paper~\cite{CGGN20}, readily imply that if $X$ is a compact \kahler space with klt singularities, smooth in codimension two and $\cc1X = 0$, then there exists an admissible Hermite--Einstein metric $h$ on $\T{\Reg X}$.
Assuming additionally that $\cc2X \cdot \alpha^{n-2} = 0$ for some \kahler class $\alpha$, this \emph{should} imply that $(\T{\Reg X}, h)$ is hermitian flat and therefore $X$ is a torus quotient by~\cite[Theorem~D]{CGGN20}.
There are at least two reasons why \cref{thmA} is still not an immediate consequence of~\cite{CW21} and~\cite{CGGN20}:
\begin{itemize}
\item In~\cite{CW21} the authors make the quite strong assumption that there is a global embedding $X \inj M$ into a compact \kahler manifold $M$, and that the \kahler metric on $X$ extends to $M$.
This is used in order to define the second Chern class and also to relate Chern numbers to integrals of Chern forms against powers of \kahler metrics.
We do not know any natural condition that would guarantee the existence of such an embedding.
\item A delicate point is to prove that $\cc2X \cdot \alpha^{n-2} = \int_{\Reg X} \cc2{\T{\Reg X}, h} \wedge \omega^{n-2}$ if $\omega$ is a \kahler metric representing $\alpha$.
This is currently not completely clear to us.
\end{itemize}

\subsection*{Open problems}

In \cref{sec open}, we have collected some natural conjectural generalizations of \cref{thmA}.
These concern group actions on complex tori that are not necessarily free in codimension two.
In this case one needs to consider a different notion of second Chern class, as well as pairs in the sense of the Minimal Model Program.

\subsection*{Acknowledgements}

We would like to thank St\'ephane Druel for bringing the paper~\cite{Shokurov92} to our attention, in connection with \cref{conj implication}.
We also thank the referee for her/his careful reading and for her/his questions, helping us to improve the content of this article.
The second author would like to thank Philipp Naumann, Mihai P\u aun and Thomas Peternell for inspiring discussions.

\section{Chern classes on singular spaces} \label{section chern classes}

We will use the following notions of Chern classes on singular spaces.
In what follows, $X$ denotes an $n$-dimensional connected normal compact complex space.

\begin{defi}[Chern classes of sheaves] \label{chern E}
Let $\sE$ be a torsion-free coherent sheaf on $X$, and let $f \from Y \to X$ be a resolution such that $f^\sharp \sE \defn \factor{f^* \sE}{\tor(f^* \sE)}$ is locally free.
For any number $1 \leq i \leq n$ and any class $a \in \HH2n-2i.X.\R.$, we set
\[ \cc i\sE \cdot a \defn \cc i{f^\sharp \sE} \cdot f^*(a) \in \HH2n.Y.\R. = \R. \]
\end{defi}

The Chern classes $\cc i\sE$ thus defined are elements of $\HH2n-2i.X.\R.\dual = \Hh2n-2i.X.\R.$.
For the existence of a resolution with the required property, see~\cite[Theorem~3.5]{Rossi68}.
Well-definedness follows from the simple observation that if $h \from Z \to X$ is a resolution which factors as $Z \xrightarrow{\;g\;} Y \xrightarrow{\;f\;} X$, then there is an exact sequence
\[ g^* \big( \!\tor(f^* \sE) \big) \lto h^* \sE \lto g^* \big( f^\sharp \sE \big) \lto 0, \]
hence $h^\sharp \sE = g^* \big( f^\sharp \sE \big)$ and consequently
\[ \cc i{h^\sharp \sE} \cdot h^*(a) = g^* \cc i{f^\sharp \sE} \cdot g^*(f^* a) = \cc i{f^\sharp \sE} \cdot f^*(a). \]

\begin{rem}[Polynomial combinations of Chern classes] \label{comb chern}
The above definition applies more generally to weighted homogeneous polynomials in the Chern classes, where $\mathrm c_i$ has degree~$i$.
For example, we may set $\cpc12\sE \cdot a \defn \cpc12{f^\sharp \sE} \cdot f^*(a)$ and then $\cpc12\sE$ will be an element of $\Hh2n-4.X.\R.$.
Note that we cannot directly define ``$\cpc12\sE \defn \cc1\sE \cdot \cc1\sE$'' because in general there is no ring structure on the homology $\Hh*.X.\R.$.

In a similar vein, we can also define combinations of Chern classes of different sheaves.
E.g.~if $\sE, \sF$ are torsion-free, we pick a resolution $f \from Y \to X$ such that both $f^\sharp \sE$ and $f^\sharp \sF$ are locally free.
Then $\cc1\sE \cdot \cc1\sF \in \Hh2n-4.X.\R.$ is defined by setting $\cc1\sE \cdot \cc1\sF \cdot a \defn \cc1{f^\sharp \sE} \cdot \cc1{f^\sharp \sF} \cdot f^*(a)$.
\end{rem}

\begin{rem-plain}
While the above definitions are sufficient for our purposes, they turn out to be rather moot for general coherent sheaves.
In particular, they do not behave nicely in short exact sequences.
For example, the Chern classes of a torsion sheaf would obviously all vanish.
\end{rem-plain}

In the case where $X$ is klt and $\sE = \T X$ is the tangent sheaf, there is another reasonable way to define the second Chern class.

\begin{defi}[Second Chern class of $X$] \label{chern X}
Assume that $X$ has klt singularities, and let $f \from Y \to X$ be a resolution which is minimal in codimension two. For any class $a \in \HH2n-4.X.\R.$, we set
\[ \cc2X \cdot a \defn \cc2Y \cdot f^*(a) \in \HH2n.Y.\R. = \R. \]
\end{defi}

\noindent
For existence of such a resolution and well-definedness of $\cc2X$, see~\cite[Proposition~5.3]{GK20}.
Again, we have $\cc2X \in \HH2n-4.X.\R.\dual = \Hh2n-4.X.\R.$.

\begin{rem-plain}
The klt assumption in \cref{chern X} can be weakened.
In fact, the proof of independence of $Y$ does not use it, and hence the only question is whether a resolution minimal in codimension two actually exists.
This is true e.g.~whenever $X$ has klt singularities in codimension two (equivalently, quotient singularities in codimension two) and thus in particular if $X$ is smooth in codimension two. 
\end{rem-plain}

The next result, together with \cref{exa c2}, clarifies the relationship between Definitions~\lref{chern E} and \lref{chern X}.

\begin{prop}[Compatibility, I] \label{comp II}
If $X$ has klt singularities and is smooth in codimension two, i.e.~$\codim X{\sing X} \geq 3$, then we have $\cc2X = \cc2{\T X}$ as elements of $\Hh2n-4.X.\R.$.
\end{prop}

\begin{proof}
Let $f \from Y \to X$ be a resolution such that $f^\sharp \T X$ is locally free.
It is clear from the construction in~\cite{Rossi68} that $f$ can be chosen to be an isomorphism over $\Reg X$, since there $\T X$ is already locally free.
As $X$ is assumed to be smooth in codimension two, such a resolution~$f$ will automatically be minimal in codimension two.
So it is sufficient to show that
\[ \cc2{f^\sharp \T X} \cdot f^* a = \cc2Y \cdot f^* a \]
for all $a \in \HH2n-4.X.\R.$.
This follows from \cref{637} below applied with $Z = \Exc(f)$, once we know that $(f^* a)\big|_Z = 0$.
But $(f^* a)\big|_Z = f^* \big( a|_{f(Z)} \big) = 0$ because $f(Z) = \sing X$ and $\HH2n-4.\sing X.\R. = 0$ by dimension reasons (the real dimension is at most $2n - 6$).
\end{proof}

Going back to the original setup (where $X$ is only assumed to be normal), recall that the determinant of the rank $r$ torsion-free sheaf $\sE$ is defined as $\det \sE \defn \big( \bigwedge^r \sE \big)\ddual$, where $(-)\ddual$ denotes the reflexive hull (= double dual).
By definition, this is a rank~one reflexive sheaf.
We are interested in situations where it is actually \Q-Cartier.
If this is the case, we may as usual consider $\cc1{\det \sE} \in \HH2.X.\R.$, and more generally $\cpc1k{\det \sE} \in \HH2k.X.\R.$.
By abuse of notation, we will also consider
\[ \cpc1k{\det \sE} \in \Hh2n-2k.X.\R. \]
via the natural map $\HH2k.X.\R. \to \HH2n-2k.X.\R.\dual = \Hh2n-2k.X.\R.$ given by the cup product pairing (or equivalently, by cap product with the fundamental class of $X$).
We compare this notion to \cref{chern E}:

\begin{prop}[Compatibility, II] \label{comp I}
Assume that $X$ is normal and smooth in codimension $k \geq 1$, that $\sE$ is locally free in codimension $k$, and that $\det \sE$ is \Q-Cartier.
Then we have
\[ \cpc1\ell\sE = \cpc1\ell{\det \sE} \qquad \text{for any $\ell \leq k$} \]
as elements of $\Hh2n-2\ell.X.\R.$.
In particular, if $X$ is smooth in codimension two, $\sE$ is reflexive, and $\det \sE \isom \O X$, then $\cpc12\sE = 0$.
\end{prop}

\begin{proof}
Let $f \from Y \to X$ be a resolution such that $f^\sharp \sE$ is locally free.
As before, we may assume that $f$ is an isomorphism over the locus where both $X$ is smooth and $\sE$ is locally free.
Pick $m > 0$ such that $\sL \defn \big( (\det \sE)^{\tensor m} \big) \ddual$ is a line bundle.
Then we need to show that
\[ \cpc1\ell{\det f^\sharp \sE} \cdot f^* a = \frac1{m^\ell} \, \cpc1\ell{f^* \sL} \cdot f^* a \]
for all $a \in \HH2n-2\ell.X.\R.$.
This follows from \cref{637} applied with $Z = \Exc(f)$, since $(\det f^\sharp \sE)^{\tensor m}$ and $f^* \sL$ are isomorphic outside of $Z$ and $(f^* a)\big|_Z = 0$ as in the proof of \cref{comp II}.

For the second statement, it suffices to show that $\sE$ is locally free in codimension two.
After discarding $\sing X$, this follows from the fact that a reflexive sheaf on a \emph{smooth} space enjoys this property~\cite[Lemma~1.1.10]{OSS80}.
\end{proof}

\begin{exa} \label{exa c2}
Without assuming that $X$ is smooth in codimension two, the Chern classes $\cc2X$ and $\cc2{\T X}$ are in general different.
For instance let $f \from Y \to X$ be the minimal resolution of a Kummer surface $X = \factor A{\pm 1}$, where $A$ is a complex $2$-torus, with exceptional divisor $E$.
The surface $Y$ being K3, we have $\cc2X = \cc2Y = 24$.
On the other hand, we have $f^\sharp \T X = \T Y(-\log E)$ as shown by the computations below.
The dual of this sheaf sits inside the residue sequence
\[ 0 \lto \Omegap Y1 \lto \Omegap Y1(\log E) \lto \O E \lto 0 \]
and then quite generally, a Chern class computation shows
\[ \cc2{\Omegap Y1(\log E)} = \cc2Y + K_Y \cdot E + E^2. \]
In our situation, this means that
\[ \cc2{\T X} = \cc2{\T Y(-\log E)} = 24 + 0 - 32 = -8 \ne 24 = \cc2 X, \]
as $E$ is the disjoint union of sixteen $(-2)$-curves.
The same example also shows that \cref{comp I} fails if $\ell > k$.
Indeed, the sheaf $\sE = \T X$ has trivial determinant, hence $\cpc12{\det \T X} = 0$, but
\[ \cpc12{\T X} = \cpc12{\T Y(-\log E)} = \big( \! - \! (K_Y + E) \big) ^2 = E^2 = -32 \ne 0. \]

Let us finally justify the equality $f^\sharp \T X = \T Y(-\log E)$.
The claim being local, we can assume that $X = \factor{\C^2}{\pm1}$ and we denote by $\mu \from \wh Z \to Z$ the blow-up of the origin in $Z = \C^2$ with coordinates $u, v$.
The $\Z/2$-action lifts to $\wh Z$ and we get a commutative diagram:
\[\xymatrix{
\wh Z \ar[d]_-p \ar[rr]^-\mu   & & Z \ar[d]^-q \\
Y = \factor{\wh Z}{\pm1} \ar[rr]^-f & & X = \factor Z{\pm1}
} \]
In one chart of the blow-up, the map $\mu$ is given by $\mu(x, y) = (xy, y)$ with exceptional divisor $\wh E = \set{y = 0}$ and the $\Z/2$-action is given by $(-1) \cdot (x, y) = (x, -y)$.
In particular, in this chart, coordinates on $Y$ are given by $(s, t) = (x, y^2)$ and $E = \set{t = 0}$.

We now notice that the $2$-form $\omega \defn \dd u \wedge \dd v$ on $Z$ is $\Z/2$-invariant and hence descends to $X$, where it induces an isomorphism $\T X \isom \Omegar X1$ by contraction.
The same is true for $\mu^*(\omega)$: it descends to $\omega_Y$ a nowhere vanishing $2$-form on $Y$, namely $\dd s \wedge \dd t$ up to a nonzero constant.
Hence $\omega_Y$ induces an isomorphism $\T Y \isom \Omegap Y1$.
Since $\big( q_* \Omegar Z1 \big)^{\Z/2}$ is reflexive by~\cite[Lemma~A.4]{GKKP}, quasi-\'etalit\'e of $q$ implies that pullback of reflexive forms along $q$ induces an isomorphism $\Omegar X1 \bij \big( q_* \Omegar Z1 \big)^{\Z/2}$.
The latter sheaf is generated by the (images of the) $1$-forms $u \, \dd u, \, u \, \dd v, \, v \, \dd u$ and $v \, \dd v$; their pull-backs to $\wh Z$ are
{ \everymath={\displaystyle}
\[ \begin{array}{ll}
xy^2 \, \dd x + x^2y \, \dd y = p^* \! \left( st\, \dd s + \frac{s^2}2 \, \dd t \right), & xy \, \dd y = p^* \! \left( \frac s2 \, \dd t \right), \\[3ex]
y^2 \, \dd x + xy \, \dd y = p^* \! \left( t \, \dd s + \frac s2 \, \dd t \right), & y \, \dd y = p^* \! \left( \frac12 \, \dd t \right).
\end{array} \] }%
The sheaf $f^\sharp \Omegar X1 \subset \Omegap Y1$ is generated by the terms in brackets, and thus by the $1$-forms $\dd t$ and $t \, \dd s$.
Consider the following commutative diagram:
\[ \xymatrix{
f^* \T X \ar[rr] \ar[d]_-{\rotatebox{90}{$\sim$}}^-{\omega} & & \T Y \ar[d]_-{\rotatebox{90}{$\sim$}}^-{\omega_Y} \\
f^* \Omegar X1 \ar[rr] & & \Omegap Y1,
} \]
The vertical arrows are isomorphisms, given by contraction as explained above.
The upper horizontal arrow is obtained by adjunction from the map $\T X \bij f_* \T Y$~\cite[Proposition~1.2]{Wahl75}.
The lower horizontal arrow is obtained similarly from $\Omegar X1 \bij f_* \Omegap Y1$, cf.~e.g.~\cite[Theorem~4.2]{GKKP}.
The map $f^* \T X \to \T Y$ is generically injective (it is an isomorphism outside of~$E$).
So its kernel is the torsion subsheaf of $f^* \T X$, and consequently its image its nothing but $f^\sharp \T X$.
Similar remarks apply to the bottom map $f^* \Omegar X1 \to \Omegap Y1$.
In order to obtain generators of $f^\sharp \T X$, it is therefore sufficient to contract the generators of $f^\sharp \Omegar X1$ determined above by~$\omega_Y$.
Doing so yields that
\[ f^\sharp \T X = \left\langle \frac{\partial}{\partial s}, \, t \frac{\partial}{\partial t} \right\rangle = \T Y(-\log E) \]
in the given chart.
The proof is finished using similar computations in the second affine chart of $\wh Z$.
\end{exa}

\begin{rem-plain}
We have only used the following facts about the singularities of $X$:
they are quotient singularities, they are Gorenstein and they admit a crepant resolution.
It should hence be possible to extend the above argument e.g.~to arbitrary canonical surface singularities.
We do not need this level of generality here.
\end{rem-plain}

\begin{lem} \label{637}
Let $X$ be a compact complex manifold of dimension $n$, $Z \subset X$ an analytic subset and $U \defn X \setminus Z$.
Suppose two locally free coherent sheaves $\sF, \sG$ on $X$ are given such that there is an isomorphism $\sF\big|_U \isom \sG\big|_U$.
Then
\[ \cc i\sF \cdot a = \cc i\sG \cdot a \qquad \text{for any $a \in \HH2n-2i.X.\R.$ with $a\big|_Z = 0$.} \]
A similar statement holds for degree $i$ polynomials in the Chern classes.
\end{lem}

\begin{proof}
Pick a sufficiently small neighborhood $Z \subset W \subset X$ which deformation retracts onto $Z$ via a map $r \from W \to Z$.
This exists because we can find a triangulation of $X$ such that $Z$ is a subcomplex~\cite[Theorems~2 and~3]{Lojasiewicz64}.
Set $W^\times \defn W \setminus Z$, and note that the inclusion $j \from (W, W^\times) \inj (X, U)$ satisfies excision.
We define a map $\varphi \from \HH2i.X.U. \to \Hhs2n-2i.Z.$ as follows (everything is with real coefficients):
\[ \HH2i.X.U. \xrightarrow{\quad j^* \quad} \HH2i.W.W^\times. \xrightarrow{\;\;\cap [W]_Z\;\;} \Hhs2n-2i.W. \xrightarrow{\quad r_* \quad} \Hhs2n-2i.Z.. \]
The first map is an isomorphism by excision.
The second map is cap product with the fundamental class\footnote{%
Recall that for an oriented $n$-manifold $M$ and a compact subset $K \subset M$, the fundamental class $[M]_K$ is the unique element of $\Hh n.M.M \setminus K.$ which at each point $x \in K$ induces the local orientation in $\Hh n.M.M \setminus \set x.$ given by the orientation of $M$ (here we are using integer coefficients).
In case $K = M$ is compact, this reduces to the usual notion.
Cf.~\cite[Lemma~3.27]{Hatcher02} for proofs.}
$[W]_Z \in \Hh2n.W.W^\times.$.
The third map $r_*$ is an isomorphism because $r$ is a deformation retraction.
(The map $\varphi$ itself is also an isomorphism, but we do not need this.
The reason is that the dual map $\varphi\dual \from \HHs2n-2i.Z. \to \Hh2i.X.U.$ is an isomorphism by~\cite[Proposition~3.46]{Hatcher02}.)

We now have a commutative diagram
\[ \xymatrix{
\HH2i.X.U. \ar^-{p^*}[rr] \ar_-\varphi[d] & & \HHs2i.X. \ar^-{\iota_U^*}[rr] \ar^-{\cap [X]}[d] & & \HHs2i.U. \\
\Hhs2n-2i.Z. \ar^-{(\iota_Z)_*}[rr] & & \Hhs2n-2i.X.
} \]
where the upper row is the long exact sequence of relative cohomology.
To see commutativity, first note that under the isomorphism $\Hh2n.W.W^\times. \isom \Hh2n.X.U.$, the fundamental classes $[W]_Z$ and $[X]_Z$ correspond to each other.
Now pick an arbitrary class $\alpha \in \HH2i.X.U.$ and calculate
\begin{align*}
(\iota_Z)_* \varphi(\alpha) & = (\iota_Z)_* r_* \big( j^* \alpha \cap [W]_Z \big) && \text{by definition} \\
& = j_* \big( j^* \alpha \cap [W]_Z \big) && \text{$r$ is a deformation retraction} \\
& = \alpha \cap j_* [W]_Z && \text{projection formula} \\
& = \alpha \cap [X]_Z && \text{remark above} \\
& = \alpha \cap p_* [X] && \text{where $p_* \from \Hhs2n.X. \to \Hh2n.X.U.$} \\
& = p^* \alpha \cap [X] && \text{by definition.}
\end{align*}

Let us now return to the original setting.
The assumption on $\sF$ and $\sG$ clearly implies that $\iota_U^*(\cc i\sF) = \iota_U^*(\cc i\sG)$, hence the difference $c \defn \cc i\sF - \cc i\sG$ lies in the image of $\HH2i.X.U.$.
By the above diagram, $c \cap [X] = (\iota_Z)_* \sigma$ for some $\sigma \in \Hhs2n-2i.Z.$.
The claim now follows easily:
\[ c \cdot a = \big\la c \cap [X], a \big\ra = \big\la (\iota_Z)_* \sigma, a \big\ra = \big\la \sigma, \iota_Z^* \, a \big\ra = 0, \]
since $\iota_Z^* \, a = a\big|_Z = 0$ by assumption.
Here $\la -, - \ra$ denotes the natural pairing between homology and cohomology.
The proof for arbitrary polynomials in the Chern classes is the same, and hence is omitted.
\end{proof}

\begin{rem-plain}
In \cref{637}, we do \emph{not} require the existence of a global map $\sF \to \sG$ whose restriction to $U$ is an isomorphism.
This greatly simplifies the proof of \cref{comp II}.
Indeed, as the sheaf of \kahler differentials $\Omegap X1$ is not locally free, from the natural map $f^* \Omegap X1 \to \Omegap Y1$ we cannot by dualizing deduce a map $\T Y \to f^* \T X$.
If we choose~$f$ to be the \emph{functorial} resolution, there is indeed a map $f^* \T X \to \T Y$, cf.~e.g.~\cite[Fact~2.2]{LowGenusLZ} and the discussion thereafter.
But then we do not know whether $f^\sharp \T X$ is always locally free.
\end{rem-plain}

\subsection*{Slopes and stability}

If $\alpha \in \HH2.X.\R.$ is a \kahler (or merely big and nef) class on $X$, one can define the \emph{slope} (with respect to $\alpha$) of a nonzero torsion-free sheaf $\sE$ as
\[ \mu_\alpha(\sE) \defn \frac1{\rk \sE} \, \cc1\sE \cdot \alpha^{n-1}. \]
In case $\det \sE$ is a \Q-line bundle (i.e.~when there exists $N \in \N^+$ such that $(\det \sE)^{[N]}$ is locally free), one recovers the usual definition of the slope (cf.~e.g.~\cite{GKP14}) thanks to \cref{comp I}.
We say that $\sE$ is \emph{(slope) stable} with respect to $\alpha$ if for any nonzero coherent subsheaf $\sF \subset \sE$ of rank strictly less than $\rk \sE$, we have $\mu_\alpha(\sF) < \mu_\alpha(\sE)$.

As before, let $f \from Y \to X$ be a resolution such that $f^\sharp \sE$ is locally free.
Given a subsheaf $\sF \subset \sE$ of rank $s > 0$, $f^\sharp \sF$ will in general not be locally free, but its determinant is a line bundle since $Y$ is smooth.
Thanks to the observations above, we can still compute the slope of $\sF$ as $\mu_\alpha(\sF) = \frac1s \, \cc1{f^\sharp \sF} \cdot f^*\alpha^{n-1}$.

\section{The Bogomolov--Gieseker inequality} \label{sec stability}

In this section, we prove \cref{BG ineq} and give an application to varieties with vanishing first Chern class (\cref{positivity c2}).

\begin{setup} \label{setup stability}
Let $X$ be a normal compact \kahler space of dimension $n$, let $\alpha\in \HH2.X.\R.$ be a \kahler class represented by a \kahler metric $\omega$.
Recall that by definition, this means that $\om$ is a genuine \kahler metric on $\Reg X$ such that for some (or equivalently any) covering $X = \bigcup_{\alpha \in A} U_\alpha$ by open sets admitting an embedding $j_\alpha \from U_\alpha \inj \C^N$ for some integer $N$, the smooth form $(j_{\alpha})_* \big( \omega|_{U_{\alpha}^{\rm reg}} \big)$ is the restriction of a \kahler metric defined on an open neighborhood of $j_{\alpha}(U_{\alpha}) \subset \C^N$.
Moreover, any such metric induces a cohomology class $[\omega]\in \HH2.X.\R.$, cf.~e.g.~\cite[\S~3.A]{GK20}.

\noindent
Next, let $\sE$ be a reflexive coherent sheaf of rank $r$ on $X$. 
In particular, the locus
\[ Z \defn X_{\rm sing} \cup \{ x\in X; \,\,\sE \mbox{ is not locally free near } x\} \]
satisfies $\codim XZ \ge 2$; one sets $X^\circ \defn X \setminus Z$.
Once and for all, we fix a log resolution $f \from \wX \to X$ of $(X, Z)$ such that $\wE \defn f^\sharp \sE = \factor{f^*\sE}{\tor f^* \sE}$ is locally free.
The exceptional divisor of $f$ is denoted by $F = \sum F_i$.
Finally, we set
\[ \Delta(\sE) \cdot \alpha^{n-2} \defn \left( 2r \, \cc2\sE - (r - 1) \, \cpc12\sE \right) \cdot \alpha^{n-2} \]
to be the Bogomolov--Gieseker discriminant of $\sE$ against the class $\alpha^{n-2} \in \HH2n-4.X.\R.$.
Let us emphasize that in this section, we do not assume that $X$ is smooth in codimension two, unless explicitly mentioned otherwise.
\end{setup}

\subsection{Openness of stability}

The goal of this subsection is to prove an openness-type property for stable sheaves on a resolution of singularities, cf.~\cref{open}.
This was previously observed in the smooth case in~\cite[Proposition~2.1]{Cao13}.
As an immediate consequence, we obtain~\lref{B.1} from the introduction.

\begin{lem} \label{elementary}
In \cref{setup stability} above, $\sE$ is stable with respect to $\alpha$ if and only if $\wE$ is stable with respect to $f^*\alpha$. 
\end{lem}

\begin{proof}

Assume that $\sE$ is stable and let $\widehat {\mathscr G}\subset \wE$ be a proper subsheaf.
Note that by reflexivity of $\sE$, the natural morphism $\sE\to f_* f^* \sE $ induces an isomorphism $(f_*f^*\sE)^{**} \isom \sE$.
Similarly, taking the direct image of the canonical surjection $f^*\sE\to \wE$ and dualizing induces an isomorphism $(f_*\wE)^{**} \isom \sE$.
Now the subsheaf $\sG \defn (f_*\widehat \sG)^{**} \subset (f_*\wE)^{**} \isom \sE$ is such that $\cc1{f^\sharp \sG} - \cc1{\widehat \sG}$ is a linear combination of the $F_i$, and therefore
\begin{equation*}
\mu_{f^*\alpha}(\widehat \sG)=\mu_{\alpha}(\sG)<\mu_{\alpha}(\sE)=\mu_{f^*\alpha}(\wE).
\end{equation*}
In the other direction, assume that $\wE$ is stable with respect to $f^*\alpha$ and let $\sF\subset \sE$ be a proper subsheaf.
The image of $f^* \sF \to f^\sharp \sE$ yields a subsheaf $\wF$ of $\wE$ which coincides with $f^*\sF$ away from $F$.
In particular, $\cc1{f^\sharp \sF} - \cc1\wF$ is supported on $F$, hence
\[\mu_{\alpha}(\sF)=\mu_{f^*\alpha}(\wF)<\mu_{f^*\alpha}(\wE)=\mu_{\alpha}(\sE)\]
and the lemma follows.
\end{proof}

\begin{lem} \label{stable}
In \cref{setup stability} above, assume that $\sE$ is stable with respect to $\alpha$. Let $\beta$ be a \kahler class on $\wX$. Then, 
\begin{enumerate}
\item The sheaf $\wE$ is stable with respect to $f^*\alpha$. 
\item \label{bounded} There exists a constant $C>0$ such that for any subsheaf $\wF\subset \wE$ and any integer $k\in [0,n-1]$, one has 
$$\cc1\wF \cdot (f^*\alpha)^{k} \cdot \beta^{n-1-k} \le C.$$
\item \label{max} There exists $\delta>0$ such that for every subsheaf $\wF\subset \wE$ with $\rk(\wF)<\rk(\wE)$, we have
$$\mu_{f^*\alpha}(\wF)\le \mu_{f^*\alpha}(\wE)-\delta.$$
\end{enumerate}
\end{lem}

\begin{proof}
The first item is an immediate consequence of \cref{elementary}. 

The second item can be proved in a similar way as the case of a single \kahler polarization, see e.g.~\cite[Chapter~V, Lemma~7.16]{Koba}. Let us provide the main ideas. First of all, since saturation increases the slope, one can assume that $\wF$ is saturated in $\wE$. In particular, $\wF\subset \wE$ corresponds to an inclusion of vector bundles $\widehat F \subset \widehat E$ on a big open set $U\subset \wX$, i.e. $\wX\setminus U$ has codimension at least two.

Next, we pick a \kahler metric $\widehat \omega\in \beta$ and a smooth hermitian metric $h_{\widehat E}$ on $\widehat E$; it induces a hermitian metric $h_{\widehat F}$ on $\widehat F|_U$ and their respective Chern curvature forms satisfy
\begin{equation}
\label{griffiths}
i\Theta(\widehat F,h_{\widehat F}) \le \mathrm{pr}_{\widehat F} i\Theta(\widehat E,h_{\widehat E})|_{\widehat F} \qquad \mbox{on } U,
\end{equation}
where $\mathrm{pr}_{\widehat F}$ is the orthogonal projection onto $\widehat F$ with respect to $h_{\widehat E}$. 
Taking the (endomorphism) trace of \eqref{griffiths} and wedging with $ f^*\omega^{k} \wedge \widehat \omega^{n-1-k}$, we get 
\begin{equation}
\label{ineq CC}
\cc{1}{\widehat F,h_{\widehat F}}\wedge f^*\omega^{k} \wedge \widehat \omega^{n-1-k} \le C \|i\Theta(\widehat E,h_{\widehat E})\|_{h_{\widehat E}, \widehat \omega} \cdot \widehat \omega^n \qquad \mbox{on } U,
\end{equation}
where $C$ depends only on $\mathrm{tr}_{\widehat \omega}(f^*\omega)$. The right-hand side does not depend on $\wF $ anymore; hence its integral over $U$ (or equivalently over $\wX$) is bounded independently of $\wF$. 

Finally, using a log resolution of $(\wX,\wX\setminus U)$, one can compute the integral of the left-hand side of \eqref{ineq CC} over $U$ and see that it is convergent, and coincides with $\cc{1}{\wF}\cdot (f^*\alpha)^{k} \cdot \beta^{n-1-k}$ since the codimension of $\wX\setminus U$ is at least two, cf.~\cite[eq.~($\ast\ast$) on p.~181]{Koba}.
This proves the claim in the second item. 

The last item is a consequence of the proof of \cref{elementary} and the fact that the statement is true for subsheaves $\sF$ of $\sE$ with rank strictly less than $\mathrm{rank}(\sE)$. As for the latter fact, it derives e.g.~from the finiteness of components of the Douady space of quotients of $\sE$ with slope at most $\mu_{\alpha}(\sE)+1$, cf.~e.g.~\cite[Corollary~6.3]{Toma} applied to $S=\mathrm{pt}$.
\end{proof}

\begin{prop}[Bogomolov--Gieseker inequality] \label{open}
In \cref{setup stability} above, assume that $\sE$ is stable with respect to $\alpha$. Then, there exists $\ep_0>0$ such that the sheaf $\wE$ is stable with respect to $f^*\alpha+\ep\beta$ for any $0\le \ep\le \ep_0$. 
In particular, one has the Bogomolov--Gieseker inequality
\[ \Delta(\sE)\cdot \alpha^{n-2}= \left(2r\cc2\sE-(r-1)\cpc12\sE\right)\cdot \alpha^{n-2} \ge 0. \]
\end{prop}

\begin{proof}
With the notation of \cref{stable} above, let us set $\ep_0:=\frac{\delta}{2(n-1)C}$. Then, one has for $\ep\le \ep_0$ and any subsheaf $\wF\subset \wE$ of rank strictly less than $\rk(\sE)$:
\begin{align*}
\mu_{f^*\alpha+\ep\beta}(\wF)&=\mu_{f^*\alpha}(\wF)+ \frac 1 r\sum_{k=0}^{n-2} \ep^{n-1-k} \cc1\wF\cdot f^*\alpha^k \cdot \beta^{n-1-k}\\
&\le \mu_{f^*\alpha}(\wE)-\delta + (n-1)C\ep\\
&\le \mu_{f^*\alpha}(\wE)-\delta/2,
\end{align*}
where the first inequality follows from~\lref{bounded} and~\lref{max} in \cref{stable}.

For the second assertion, one can use the Kobayashi--Hitchin correspondence to obtain for any $\ep>0$ a Hermite--Einstein metric $h_{\ep}$ on $\wE$ with respect to a \kahler metric $\om_\ep$ of the form $\om_\ep:=f^*\om+\ep \widehat \om$ where $\om\in \alpha$ (resp. $\widehat \om\in \beta$) is a \kahler metric on $X$ (resp. on $\wX$). It is classic to get the inequality \[\left(2r\cc2{\wE,h_\ep}-(r-1)\cc{1}{\wE,h_\ep}^2 \right)\wedge \om_\ep^{n-2} \ge 0\] pointwise, for any $\ep>0$. Integrating over $\wX$ and letting $\ep$ go to zero, one gets the expected inequality. 
\end{proof}

\subsection{Bogomolov--Gieseker inequality for semistable sheaves}

In this section, we explain how to extend \cref{open} to the semistable case.
This is quite likely standard, but we will give the details for the reader's convenience.
The results in this section are strictly speaking not needed in the rest of the article, but they do simplify the proof of \cref{positivity c2} somewhat.

\begin{prop}[Bogomolov--Gieseker inequality for semistable sheaves] \label{BG ss}
In \cref{setup stability} above, assume that $\sE$ is semistable with respect to $\alpha$.
Then
\[ \Delta(\sE) \cdot \alpha^{n-2} \ge 0. \]
\end{prop}

The natural strategy is to consider the Jordan--H\"older filtration of $\sE$ and analyse how the discriminant behaves under a short exact sequence.
This is the content of \cref{JH ex} and \cref{BG ES} below.
If $\gamma$ is a \kahler class, then \cref{JH ex} is contained in~\cite[Chapter~V, Theorem~7.18]{Koba}, but without a proof.
A statement similar to \cref{BG ES} can be found in~\cite[Corollary~7.3.2]{HuybrechtsLehn10}.

\begin{lem}[Jordan--H\"older filtrations] \label{JH ex}
Let $X$ be an $n$-dimensional compact \kahler manifold, $\gamma \in \HH1,1.X.\R.$ a nef and big class and $\sE$ a $\gamma$-semistable reflexive sheaf.
Then $\sE$ admits a \emph{Jordan--H\"older filtration}, i.e.~a filtration
\[ 0 = \sE_0 \subset \sE_1 \subset \dots \subset \sE_d = \sE \]
where for each $i = 0, \ldots, d - 1$, the sheaf $\Gr_i \sE \defn \factor{\sE_{i+1}}{\sE_i}$ is torsion-free and $\gamma$-stable with $\mu_\gamma \big( \! \Gr_i \sE \big) = \mu_\gamma(\sE)$.
\end{lem}

\begin{proof}
Consider the set of all filtrations $\sE_\bullet$ of $\sE$ whose graded pieces are torsion-free and semistable of slope $\mu_\gamma(\sE)$.
This set is nonempty and partially ordered by refinement.
Furthermore, if $\sE$ is not stable, then there is a proper subsheaf $\sE'$ with $\mu_\gamma(\sE') = \mu_\gamma(\sE)$.
Clearly both $\sE'$ and $\factor{\sE}{\sE'}$ are semistable.
This means that any filtration having a non-stable graded piece can be refined.
Consequently, a filtration maximal with respect to refinement (which exists due to rank reasons) is a Jordan--H\"older filtration.
\end{proof}

\begin{lem}[Discriminant of extension] \label{BG ES}
Let $X$ be a compact \kahler manifold of dimension $n$, and let $\gamma \in \HH1,1.X.\R.$ be a nef and big class.
Assume that we have a short exact sequence
\[0 \lto E_1 \lto E \lto E_2 \lto 0\]
of reflexive sheaves such that $\mu_\gamma(E_1) = \mu_\gamma(E_2)$ and $\Delta(E_i) \cdot \gamma^{n-2} \ge 0$ for $i =1, 2$.
Then,
\[ \Delta(E) \cdot \gamma^{n-2} \ge 0. \]
\end{lem}

\begin{proof}
Let $r_i=\rk(E_i)$, $\alpha_i=c_1(E_i)$, $\beta_i=c_2(E_i)$ (resp. $r=\rk(E)$, $\alpha=c_1(E)$, $\beta=c_2(E)$. We have
\[r=r_1+r_2, \quad \alpha=\alpha_1+\alpha_2, \quad \beta=\beta_1+\beta_2+\alpha_1\alpha_2.\]
Therefore, we get
\begin{align*}
2r\beta &= 2r_1\beta_1+2r_2\beta_2+2r\alpha_1\alpha_2+2r_2\beta_1+2r_1\beta_2\\
(r-1)\alpha^2 &= (r_1-1)\alpha_1^2+(r_2-1)\alpha_2^2+2(r-1)\alpha_1\alpha_2+r_2\alpha_1^2+r_1\alpha_2^2.
\end{align*}
Rewriting the inequality $\Delta(E_i) \cdot \gamma^{n-2} \ge 0$ as  
\[2r_i \beta_i \cdot \gamma^{n-2} \ge (r_i-1) \alpha_i^2 \cdot \gamma^{n-2},\] we get 
{\small
\begin{equation*}
2r\beta \cdot \gamma^{n-2} \ge \left((r_1-1) \alpha_1^2 +(r_2-1) \alpha_2^2 +2r \alpha_1\alpha_2 +(r_2-1) \frac{r_1}{r_2}\alpha_2^2+  (r_1-1) \frac{r_2}{r_1}\alpha_1^2\right)\cdot \gamma^{n-2}
\end{equation*}}
hence
{\small 
\begin{align*}
(2r\beta-(r-1)\alpha^2) \cdot \gamma^{n-2}& \ge \left(2\alpha_1\alpha_2+\frac{r_2}{r_1}[(r_1-1)-r_1]\alpha_1^2+\frac{r_1}{r_2}[(r_2-1)-r_2]\alpha_2^2\right) \cdot \gamma^{n-2}\\
&=\left(2\alpha_1\alpha_2-\frac{r_2}{r_1}\alpha_1^2-\frac{r_1}{r_2}\alpha_2^2 \right) \cdot \gamma^{n-2}\\
&=-r_1r_2\,\Big(\frac{\alpha_1}{r_1}-\frac{\alpha_2}{r_2}\Big)^2\cdot \gamma^{n-2}. 
\end{align*}}%
Now, since $\mu_\gamma(E_1)=\mu_{\gamma}(E_2)$, we have $\big(\frac{\alpha_1}{r_1}-\frac{\alpha_2}{r_2}\big)\cdot \gamma^{n-1}=0$, hence
\[\Big(\frac{\alpha_1}{r_1}-\frac{\alpha_2}{r_2}\Big)^2\cdot \gamma^{n-2} \le 0\]
by the Hovanskii--Teissier inequality, cf.~\cite[Proposition~2.5]{DP03}.
The lemma is proved.
\end{proof}

We can now finish the proof of \cref{BG ss}.

\begin{proof}[Proof of \cref{BG ss}]
Let $f \from \wh X \to X$ be a resolution as in \cref{setup stability} and set $\wh \alpha \defn f^*\alpha$.
The class $\wh \alpha$ is a nef and big class on the compact \kahler manifold $\wh X$.
Since $\wh \sE$ is $\wh \alpha$-semistable (cf.~the proof of~\cref{elementary}), by \cref{JH ex} it has a Jordan--H\"older filtration
\[ 0 = \wh \sE_0 \subset \wh \sE_1 \subset \dots \subset \wh \sE_d = \wh \sE \]
whose graded pieces $\Gr_i \wh\sE$ are $\wh\alpha$-stable.

\begin{claim} \label{1055}
Pick a \kahler class $\beta$ on $\wh X$.
Then for $\ep > 0$ sufficiently small, the sheaf $\Gr_i \wh\sE$ remains $(\wh\alpha + \ep \beta)$-stable.
\end{claim}

\begin{proof}[Proof of \cref{1055}]
For instance, this follows from~\cite[Proposition~2.1]{Cao13} (cf.~also the proof of~\cite[Lemma~3.2]{CaoHoering17}), but one could also appeal to \cref{stable}.
In order to do so, define inductively the saturation $\sE_{i}$ of $f_* \big( \wh \sE_{i}\big|_{\wh X \setminus F} \big)$ inside $\sE_{i+1}$, starting from $i = d - 1$ down to $i = 1$.
The arguments of \cref{elementary} show that the torsion-free sheaves $\Gr_i \sE \defn \factor{\sE_{i+1}}{\sE_i}$ are $\alpha$-stable.
By the proof of \cref{stable}, $\Gr_i \wh\sE$ remains $(\wh\alpha + \ep \beta)$-stable, which had to be shown.
\end{proof}

By \cref{1055} and the same perturbation argument as in \cref{open}, we conclude that
\[ \Delta ( \Gr_i \wh\sE ) \cdot \wh \alpha^{n-2} \ge 0. \]
Using \cref{BG ES} inductively, we get
\[ \Delta(\wh \sE) \cdot \wh \alpha^{n-2} \ge 0. \]
Since the left-hand side equals $\Delta(\sE) \cdot \alpha^{n-2}$ by definition, \cref{BG ss} is proved.
\end{proof}

\subsection{Application to varieties with trivial first Chern class} \label{subsec ICY IHS}

Let $X$ be a compact \kahler space with klt singularities such that $\cc1X = 0 \in \HH2.X.\R.$.
In that setting, the Abundance conjecture is known, i.e.~$K_X$ is a torsion \Q-line bundle, cf.~e.g.~\cite[Corollary~1.18]{JHM2}.
If we assume additionally that $X$ is smooth in codimension two, we immediately infer that $\cpc12{\T X} \cdot \alpha^{n-2} = 0$ for any \kahler class $\alpha$ thanks to \cref{comp I}.
Moreover, the recent Decomposition Theorem~\cite[Theorem~A]{BakkerGuenanciaLehn20} asserts that up to a \qe cover, $X$ splits as a product of a complex torus, irreducible Calabi--Yau varieties and irreducible holomorphic symplectic varieties, where the latter two are defined as follows.

\begin{defi}[ICY and IHS varieties] \label{defi IHS ICY}
Let $X$ be a compact \kahler space of dimension $n \ge 2$ with canonical singularities and $\can X \isom \O X$.
\begin{enumerate}
\item We call $X$ \emph{irreducible Calabi--Yau} (ICY) if $\HH0.Y.\Omegar Yp. = 0$ for all integers $0 < p < n$ and all \qe covers $Y \to X$, in particular for $X$ itself.
\item We call $X$ \emph{irreducible holomorphic symplectic} (IHS) if there exists a holomorphic symplectic two-form $\sigma \in \HH0.X.\Omegar X2.$ such that for all \qe covers $\gamma \from Y \to X$, the exterior algebra of global reflexive differential forms is generated by $\gamma^{[*]} \sigma$.
\end{enumerate}
\end{defi}

Given the Bochner principle~\cite[Theorem~A]{CGGN20}, it is relatively easy to show that the tangent sheaf $\T X$ of an IHS variety $X$ (resp.~ICY variety $X$) is stable with respect to any polarization.
The following result is then an immediate consequence of the Decomposition Theorem and \cref{open}.
However, we can give a more pedestrian proof by relying only on the polystability of $\T X$ (actually only its semistability) rather than a global structure result.

\begin{cor}[Semipositivity of $\mathrm c_2$] \label{positivity c2}
Let $X$ be a compact \kahler space with klt singularities such $\cc1X = 0 \in \HH2.X.\R.$.
Assume furthermore that $X$ is smooth in codimension two.
Then, one has
\[ \cc2X \cdot \alpha^{n-2} \geq 0 \]
for any \kahler class $\alpha \in \HH2.X.\R.$.
\end{cor}

\begin{proof}
We recalled above that $K_X$ is torsion; in particular, there exists a \qe cover $p:\widetilde X \to X$ such that $\det \T {\wh X} \isom \O {\wt X}$, hence $\cpc12{\T {\wt X}} \cdot p^*\alpha^{n-2} = 0$ by \cref{comp I}. By \cite[Proposition~5.6]{GK20}), we get 
\[\Delta(\T X) \cdot \alpha^{n-2} = \frac{1}{\mathrm{deg}(p)} \cc2{\T {\wt X}} \cdot p^*\alpha^{n-2}. \]
The corollary now follows from the polystability of $\T {\wt X}$ with respect to any Kähler class \cite[Theorem~A]{GSS} combined with \cref{BG ss}.
\end{proof}

\begin{rem-plain}
It is instructive to return to the Kummer surface $X$ of \cref{exa c2} to see how the above proof fails if $\codim X{\sing X} = 2$.
In that example, $\Delta(\T X) = 4 \cdot (-8) - (- 32) = 0$, but $\cpc12{\T X}$ is nonzero and $\cc2{\T X}$ ends up being negative.
Of course, the statement of \cref{BG ss} itself still holds in this example.
\end{rem-plain}

\begin{rem-plain}
In the spirit of~\cite[Theorem~6.6]{Miyaoka}, it would be interesting to prove the inequality $\cc2X \cdot \alpha^{n-2} \geq 0$ under the weaker assumption that $K_X$ is nef, although this is probably quite challenging.
\end{rem-plain}

\subsection{The case of equality}

In this subsection, we prove~\lref{B.2} and~\lref{B.3}.
That is, we discuss what happens if equality holds in the Bogomolov--Gieseker inequality.
\cref{change of class} below asserts that if $\cpc12\sE$ and $\cc2\sE$, seen as symmetric multilinear forms on $\HH2.X.\R.$, vanish against \emph{one} \kahler class, then they vanish against \emph{any} \kahler class.
In case $X$ has rational singularities, the Hodge structure on $\HH2.X.\C.$ is pure and the statement has a nice reformulation in Hodge-theoretic terms: the Chern classes in question vanish against $\HHs1,1.X.$.

\begin{prop}[Vanishing discriminant, I] \label{change of class}
In \cref{setup stability}, assume that $\sE$ is stable with respect to $\alpha$.
Let $\beta$ be an arbitrary \kahler class. 
\begin{enumerate}
\item\label{coc.1} If $\Delta(\sE) \cdot \alpha^{n-2} = 0$, then $\Delta(\sE) \cdot \beta^{n-2} = 0$.
\item\label{coc.2} If $\cpc12\sE \cdot \alpha^{n-2} = \cc2\sE \cdot \alpha^{n-2} = 0$, then $\cpc12\sE \cdot \beta^{n-2} = \cc2\sE \cdot \beta^{n-2} = 0$.
\end{enumerate}
\end{prop}

\begin{cor}[Vanishing discriminant, II] \label{coc rational}
In \cref{setup stability}, assume that $X$ has rational singularities and that $\sE$ is stable with respect to $\alpha$.
\begin{enumerate}
\item\label{cocr.1} If $\Delta(\sE) \cdot \alpha^{n-2} = 0$, then we have
\[ \Delta(\sE) \cdot \alpha_1 \cdots \alpha_{n-2} = 0 \]
for any $\alpha_1, \dots, \alpha_{n-2} \in \HH1,1.X.\R. \defn F^1 \HH2.X.\C. \cap \HH2.X.\R.$.
\item\label{cocr.2} If $\cpc12\sE \cdot \alpha^{n-2} = \cc2\sE \cdot \alpha^{n-2} = 0$, then
\[ \cpc12\sE \cdot \alpha_1 \cdots \alpha_{n-2} = \cc2\sE \cdot \alpha_1 \cdots \alpha_{n-2} = 0 \]
for any $\alpha_1, \dots, \alpha_{n-2} \in \HH1,1.X.\R.$.
\end{enumerate}
\end{cor}

\begin{rem}[Topological vanishing]
In \cref{coc rational}, it would be very desirable to show the vanishing on all of $\HH2.X.\R.$, if only because it would drastically simplify the proof of \cref{thmA}.
The reason is that the vanishing would then be a topological statement and hence propagate to any \lt algebraic deformation.
In particular, we would not need the full force of the Decomposition Theorem, but only the affirmative answer to the Kodaira problem~\cite[Theorem~B]{BakkerGuenanciaLehn20}.

If $X$ is smooth, the classical argument shows that both classes are actually zero as elements of $\HH4.X.\R.$ and, in particular, they are also zero as symmetric multilinear forms on $\HH2.X.\R.$.
However, we are not able to derive the latter conclusion in our setting.
For instance, we do not know about the vanishing of $\cc2\sE \cdot (\sigma + \overline\sigma)^{n-2}$ if $\sigma \in \HHs2,0.X.$ is nonzero.
The underlying difficulty here is that reflexive forms, which naturally represent classes in $\HHs p,0.X.$, may not be smooth, i.e.~they may not be the restriction of smooth forms under local embeddings $X\underset{\rm loc}{\hookrightarrow} \mathbb C^N$.
As a result, we do not have in general
\[ (\sigma + \overline\sigma)^{n-2} \lesssim \omega^{n-2} \]
if $\sigma$ is a reflexive $2$-form and $\omega$ a \kahler metric, thus preventing the argument below from going through.
\end{rem}

\begin{proof}[Proof of~\cref{change of class}]
Denote by $\widehat E$ the vector bundle on $\wX$ associated to $\wE$, and let $\omega \in \alpha$ (resp. $ \omega' \in \beta$) be a \kahler metric on $X$. As in the proof of \cref{open}, we pick an Hermite--Einstein metric $h_\ep$ on $E$, i.e.  
 $$i\Theta(\widehat E,h_{\ep}) \wedge \ome^{n-1} =\lambda_\ep \mathrm{Id}_{\widehat E} \, \ome^n.$$
 where $\ome=f^*\om+\ep \om_{\wX}$, for  $\om_{\wX}$ an arbitrary \kahler metric on $\wX$, and $\lambda_\ep = \frac {c_1(\widehat E) \cdotp [\ome]^{n-1}}{r [\ome]^n}$. A standard computation (see e.g.~\cite[Theorem~4.7]{Koba}) using the Hermitian--Einstein condition shows that
 $$\int_{\wX}\left(2r \cc{2}{\widehat E,h_\ep}-(r-1)\cc{1}{\widehat E,h_\ep}^2\right) \wedge \ome^{n-2} = d_n \int_{\widehat E}\|\Theta_\ep^\circ\|_{h_\ep,\ome}^2\cdot \ome^n$$
 and
 $$\int_{\wX} \cc{1}{\widehat E,h_\ep}^2 \wedge \ome^{n-2} = -d_n \int_{\wX}  \|\frac 1r\mathrm{tr}_{\rm End}(\Theta_\ep) \cdot \mathrm{Id}_{\widehat E}\|_{h_\ep,\ome}^2\cdot \ome^n$$
where $\Theta_\ep:=\Theta(\widehat E,h_{\ep})$ and $\Theta_\ep^\circ = \Theta_\ep-\frac 1r\mathrm{tr}_{\rm End}(\Theta_\ep) \cdot \mathrm{Id}_{\widehat E}$ and $d_n=\frac{1}{4\pi^2n(n-1)}$.  \\

{\bf Proof of~\lref{coc.2}.} Assume $\cc1\sE^2\cdot \alpha^{n-2}=\cc2\sE\cdot \alpha^{n-2}=0$.

\noindent
Given the assumptions above, one finds respectively  
\begin{align*}
\cc{1}{\widehat E}^2\cdot [\ome]^{n-2}&=\cc1{\widehat E}^2\cdot (f^*\alpha)^{n-2}+O(\ep)=O(\ep)\\ 
 \cc2{\widehat E}\cdot [\ome]^{n-2}&=\cc2{\widehat E}\cdot (f^*\alpha)^{n-2}+O(\ep)=O(\ep).
 \end{align*}
 Combined with the identities above, one finds that there is a constant $C_1>0$ such that 
\begin{equation}
\label{c_2 L2}
\int_{\wX} \|\Theta_\ep\|_{h_\ep,\ome}^2\cdot \ome^n \le C_1 \ep.
\end{equation}

Let us set $ \om_\ep':=f^* \omega'+\ep \om_{\wX}$. Clearly, there exists $C_2>0$ such that $C_2^{-1}\ome \le  \om_\ep' \le C_2 \ome$, which yields another constant $C_3$ satisfying
\[C_3^{-1} \|\Theta_\ep\|_{h_\ep,\ome}^2 \le  \|\Theta_\ep\|_{h_\ep, \om_\ep'}^2 \le C_3  \|\Theta_\ep\|_{h_\ep,\ome}^2.\]
Given~\lref{c_2 L2}, we find 
\begin{equation}
\label{c_2 L2'}
\lim_{\ep \to 0} \int_{\wX} \|\Theta_\ep\|_{h_\ep, \om_\ep'}^2\cdot  \om_\ep'^n =0.
\end{equation}
Now, write 
$\cc2{\widehat E}\cdot (f^*\beta)^{n-2}=\lim_{\ep \to 0} \int_{\wX} \cc2{\widehat E,h_\ep}\wedge  \om_\ep'^{n-2}$ and remember that up to some dimensional constants, one has $\cc2{\widehat E,h_\ep}=\mathrm{tr}_{\rm End}(\Theta_\ep\wedge \Theta_\ep)-\mathrm{tr}_{\rm End}(\Theta_\ep)^2$ so that
\[\left|\int_{\wX} \cc2{\widehat E,h_\ep}\wedge  \om_\ep'^{n-2}\right|\le C_4  \int_{\wX} \|\Theta_\ep\|_{h_\ep, \om_\ep'}^2\cdot  \om_\ep'^n\]
and~\lref{coc.2} follows from~\lref{c_2 L2'}. \\

{\bf Proof of~\lref{coc.1}.} Assume $\Delta(\sE) \cdot \alpha^{n-2} = 0$.

\noindent
Observe that by a standard computation, one has $\Delta(\mathrm{End}(\widehat E))=2r^2\Delta(\widehat E)$. Moreover,  $\cc1{\mathrm{End}(\widehat E)}=0$ in $\HH2.\widehat X.\R.$ so that the assumptions yield 
\begin{equation}
\label{c1 c2}
\cc1{\mathrm{End}(\widehat E)}^2\cdot f^*\alpha^{n-2}=0, \quad \mbox{and} \quad \cc2{\mathrm{End}(\widehat E)}\cdot f^*\alpha^{n-2}=0.
\end{equation}
 The Hermite--Einstein metric $h_\ep$ on $\widehat E$ with respect to $\ome$ above yields a Hermite--Einstein metric $\widetilde h_\ep$ on $\mathrm{End}(\widehat E)$.
 From~\lref{c1 c2} and the proof of~\lref{coc.2}, it follows that the curvature tensor of $\widetilde h_{\ep}$ converges to zero in $L^2$ norm with respect to $\ome$, or equivalently with respect to $\om_\ep'$.
 \lref{coc.1} now follows.
\end{proof}

\begin{proof}[Proof of \cref{coc rational}]
Let $f \from Y \to X$ be a resolution of singularities, where $Y$ is \kahler.
Since $X$ has rational singularities, we have the following diagram, where the horizontal maps are induced by multiplication with $\mathrm i = \sqrt{-1}$:
\[ \xymatrix{
\HH2.Y.\R. \ar[rr] & & \HH2.Y.\O Y. \\
\HH2.X.\R. \ar^-\beta[rr] \ar@{ ir->}^-{f^*}[u] & & \HH2.X.\O X. \ar_-{\rotatebox{90}{$\sim$}}[u]
} \]
This implies immediately that $\HH1,1.X.\R. \subset \ker \beta$.
For the other inclusion, recall that any morphism of mixed Hodge structures is strict with respect to both the Hodge and the weight filtration~\cite[Corollary~3.6]{PetersSteenbrink08}.
Pick an arbitrary element $\gamma \in \ker \beta$.
Then the $(0, 2)$-part of $f^*(\gamma)$ is zero.
Equivalently, $f^*(\gamma) \in F^1 \HH2.Y.\C.$.
By strictness, it follows that $\gamma \in F^1 \HH2.X.\C.$.
As $\gamma$ is real, we see that $\gamma \in \HH1,1.X.\R.$ as desired.

Now let $\PH X$ be the sheaf of real-valued pluriharmonic functions on $X$.
By~\cite[Proposition~6.3]{BakkerGuenanciaLehn20}, the natural map $\HH1.X.\PH X. \to \HH2.X.\R.$ is injective with image equal to $\ker \beta$ and we obtain
\[ \HH1,1.X.\R. = \HH1.X.\PH X.. \]
Recall also that the \kahler cone is open in $\HH1.X.\PH X.$ by~\cite[Proposition~3.8]{GK20}.

For the actual proof of \cref{coc rational}, we focus on~\lref{cocr.1}, since the proof of~\lref{cocr.2} is entirely similar.
By~\lref{coc.1} and the above remarks, the assumptions of~\lref{cocr.1} first imply that $\Delta(\sE) \cdot \beta^{n-2} = 0$ for any $\beta \in \HH1,1.X.\R.$.
The vanishing for possibly different classes $\alpha_1, \dots, \alpha_{n-2} \in \HH1,1.X.\R.$ then follows by a standard polarization argument~\cite{Thomas14}.
\end{proof}

\section{The second Chern class of IHS varieties} \label{sec c2 IHS}

In this section, we discuss the second Chern class of singular holomorphic symplectic varieties $X$.
The main result is as follows.
Note that here we do not assume $X$ to be smooth in codimension two.

\begin{prop}[Positivity of $\mathrm c_2$] \label{c2 IHS}
In \cref{setup IHS} below, we have $\cc2X \cdot b^{2n-2} > 0$ for any class $b \in \HH2.X.\R.$ with $q_X(b) > 0$.
In particular, this holds whenever $b$ is a \kahler class.
\end{prop}

\begin{rem-plain}
If $X$ satisfies the condition $\codim X{\sing X} \geq 3$ (which, by results of Namikawa and Kaledin, is equivalent to $\codim X{\sing X} \geq 4$), the proof of \cref{c2 IHS} can be somewhat simplified.
To be more precise, from \cref{Fujiki} we only need the existence of the Fujiki constant $C$, but not its positivity and deformation invariance.
Instead, we can obtain $C \geq 0$ from \cref{positivity c2} and $C \ne 0$ from \cref{non-vanishing c2}.
\end{rem-plain}

For the rest of this section, we work in the following setup.

\begin{setup} \label{setup IHS}
Let $X$ be an IHS variety of complex dimension $2n \geq 2$ in the sense of \cref{defi IHS ICY}.
We denote by $\sigma\in \HH0.X.\Omegar X2.$ a holomorphic symplectic $2$-form, which is unique up to a scalar.
Furthermore, we denote by $q_X \from \HH2.X.\C. \to \C$ the BBF (= Beauville--Bogomolov--Fujiki) form of $X$.
We will always normalize $q_X$ in such a way that it comes from an \emph{indivisible} integral quadratic form $\HH2.X.\Z. \to \Z$, cf.~\cite[Lemma~5.7]{BL18}.
With this convention, the BBF form is a topological invariant of $X$.
In particular, it is invariant under \lt deformations.
\end{setup}

\begin{lem}[Non-vanishing of $\mathrm c_2$] \label{non-vanishing c2}
We have $\cc2X \ne 0$ on $\HH2.X.\R.$.
That is, there exists a class $a \in \HH2.X.\R.$ such that $\cc2X \cdot a^{2n-2} \ne 0$.
\end{lem}

\begin{proof}
By~\cite[Corollary~1.4]{BL18}, $X$ admits a \lt algebraic approximation $\fX \to \Delta$, where $\Delta$ is smooth.
Let $\fY \to \fX$ be the simultaneous resolution obtained in \cite[Lemma~4.2]{GrafSchwald20}, and let $X_t$, $Y_t$ be the fibres of the respective maps.
Note that the fibrewise resolutions $Y_t \to X_t$ are then minimal in codimension two.
For any $t \ne 0$, we have a commutative diagram
\[ \xymatrix{
\HH*.Y_0.\R. & \HH*.\fY.\R. \ar_-\sim[l] \ar^-\sim[r] & \HH*.Y_t.\R. \\
\HH*.X_0.\R. \ar[u] & \HH*.\fX.\R. \ar_-\sim[l] \ar^-\sim[r] \ar[u] & \HH*.X_t.\R. \ar[u]
} \]
where the horizontal maps are isomorphisms due to the topological triviality of the \lt maps $\fY \to \Delta$ and $\fX \to \Delta$, which itself follows e.g.~from~\cite[Proposition~6.1]{AV19}.
Note that the relative tangent sheaf $\T{\fY/\Delta}$ is locally free, so we can consider its second Chern class $\cc2{\T{\fY/\Delta}} \in \HH4.\fY.\R.$.
By construction, this class gets mapped to $\cc2{Y_0}$ and $\cc2{Y_t}$, respectively, under the upper horizontal maps in the above diagram.
Since $X_0 = X$, this shows that the following conditions are equivalent:
\begin{enumerate}
\item\label{652.1} For any $a \in \HH2.X.\R.$, we have $\cc2X \cdot a^{2n-2} = 0$.
\item\label{652.2} For any $a \in \HH2.\fX.\R.$, we have $\cc2{\T{\fY/\Delta}} \cdot a^{2n-2} = 0 \in \HH4n.\fY.\R. = \R$.
\item\label{652.3} For any $a \in \HH2.X_t.\R.$, we have $\cc2{X_t} \cdot a^{2n-2} = 0$.
\end{enumerate}
We now argue by contradiction and assume that condition~\lref{652.1} is satisfied.
Let $t \in \Delta$ be such that $X_t$ is projective.
Pick an arbitrary ample divisor $H$ on $X_t$.
Then $\cc2{X_t} \cdot H^{2n-2} = 0$ by~\lref{652.3}.
Applying~\cite[Theorem~7.1]{GKP14}, we obtain a finite \qe cover $A_t \to X_t$, where $A_t$ is an abelian variety.\footnote{The cited reference makes the extra assumption that $X_t$ be \Q-factorial. However, this is not used in the proof. In fact, a general complete intersection surface $S \subset X_t$ will again have canonical singularities and in particular be rational and \Q-factorial. Therefore~\cite[Lemma~7.2]{GKP14} can be applied to $S$.}
By the argument in the proof of~\cite[Lemma~8.8]{CGGN20}, this \qe cover can be extended to a \qe cover $\fA \to \fX$ such that the induced map $\fA \to \Delta$ is again \lt.
In the situation at hand, $\fA \to \Delta$ will even be smooth since $A_t$ is smooth.
By the local constancy of Hodge numbers in smooth families, we see that
\[ \hh0.A_0.\Omegap{A_0}1. = \hh0.A_t.\Omegap{A_t}1. = 2n > 0. \]
On the other hand, $\hh0.A_0.\Omegap{A_0}1. = 0$ because $A_0 \to X_0 = X$ is a \qe cover and $X$ is~IHS.
This is the desired contradiction.
\end{proof}

The following result is an adaptation of a well-known property of Chern classes on smooth IHS manifolds, cf.~e.g.~\cite[Proposition~2.2]{OG12}.
It has to be noticed that the proof given in~\cite[Proposition~5.20]{BL18} has a different flavor.
Also, the first result in this direction (under stronger assumptions) appears to be~\cite[Lemma~2.4]{Matsushita01}.

\begin{prop}[Fujiki relations for $\mathrm c_2$] \label{Fujiki}
There exists a positive rational constant $C = C(X) \in \Q^+$, called the \emph{Fujiki constant with respect to $\cc2X$,} such that for any $a \in \HH2.X.\R.$, we have
\[ \cc2X \cdot a^{2n-2} = C \cdot q_X(a)^{n-1}. \]
Furthermore, $C(X)$ is constant in \lt families.
More precisely, if $\fX \to B$ is a \lt deformation over a (reduced and connected) base $B$, then $C(X_t) = C(X_s)$ for all $t, s \in B$.
\end{prop}

\begin{proof}
We mimic the proof of \cite{OG12}, claiming no originality.
Let $\fX \to \Delta$ be a representative of the semiuniversal family over the \lt deformation space $\Deflt(X)$.
(See~\cite{FlennerKosarew87} for the existence and construction of $\Deflt(X)$.)
The germ $\Delta$ is smooth according to~\cite[Theorem~4.7]{BL18}.
We consider the period map
\[
\pi:\left\{\begin{array}{rcl}\Delta &\lto & \Omega(X) \defn \left\{x\in\PP(\HH2.X.\C.)\mid q_X(x)=0\right\} \\
t &\mapsto & \PP(\mathrm{H}^{2,0}(X_t))
\end{array}\right.
\]
that is a local isomorphism as stated in~\cite[Proposition~5.5]{BL18}.
Let us now consider the degree $2(n-1)$ homogeneous polynomial defined by
\[ G(\alpha) \defn \cc2X \cdot \alpha^{2n-2} \]
for any $\alpha \in \HH2.X.\C.$.
Type considerations (and Gauss--Manin invariance of the second Chern class) yield that
\[
\forall\,\alpha_1,\ldots,\alpha_{n-2}\in\HH2.X.\C.,\quad \int_{X_t}\cc{2}{X_t}\wedge\sigma_t^{n}\wedge\alpha_1\wedge\cdots\wedge\alpha_{n-2}=0
\]
where $\sigma_t$ is the symplectic form on $X_t$ (for $t \in \Delta$).
The latter can be interpreted as saying that all the derivatives of $G$ up to order $n-2$ vanish along the image of $\pi$.
The Zariski closure of the image of $\pi$ being $\Omega(X)$, the zero locus of the quadratic polynomial $q_X$ (see above), we infer that $G$ has to be of the form
\begin{equation} \label{718}
G = C \cdot q_X^{n-1}
\end{equation}
with $C \in \C$ a constant.
\cref{non-vanishing c2} immediately implies $C \ne 0$, and by evaluating~\lref{718} at some $a \in \HH2.X.\Q.$ with $q_X(a) \ne 0$, we see that $C \in \Q$.
This argument also shows that $C$ remains constant in a \lt family, using the fact that $q_X$ is unchanged under such a deformation.

It remains to be seen that $C \geq 0$.
To this end, let $\fX \to \Delta$ be as above, and pick $t \in \Delta$ such that $Y \defn X_t$ is projective.
Let $H$ be an ample divisor on $Y$.
By the above observations, it suffices to show that $C(Y) \geq 0$.
Since $q_Y(H) > 0$, this is equivalent to $\cc2Y \cdot H^{2n-2} \geq 0$, which is what we will show.
By~\cite[Corollary~8.6]{Miyaoka87}, the cotangent sheaf $\Omegar Y1$ is generically nef, hence $H$-semistable, as $\cc1Y = 0$.
By~\cite[Theorem~1.2]{Flenner84}, the restriction $\Omegar Y1\big|_S$ to a general complete intersection surface $S \subset Y$ remains semistable.
Since $S$ has only quotient singularities, $\Omegar Y1\big|_S$ is automatically a \Q-vector bundle and we may apply~\cite[Lemma~2.5]{Kawamata92}.
This yields $\ccorb2Y \cdot H^{2n-2} \geq 0$, where $\ccorb2Y$ denotes the second orbifold Chern class (or \Q-Chern class) of $Y$ in the sense of~\cite{SBW94}.
We know, however, that the inequality $\ccorb2Y \cdot H^{2n-2} \leq \cc2Y \cdot H^{2n-2}$ holds in this situation:
if $\dim Y = 3$, this is~\cite[Proposition~1.1]{SBW94}.
In general, the proof is just the same, as has already been observed in~\cite[Remark~1.5]{LuTaji18}.
\end{proof}

\begin{proof}[Proof of \cref{c2 IHS}]
The first part of the statement is clear from \cref{Fujiki}.
It only remains to be seen that $q_X(b) > 0$ for any \kahler class $b \in \HH2.X.\R.$.
To this end, note that the ``usual'' degree zero Fujiki relations on $X$ yield $\int_X b^{2n} = \mu \, q_X(b)^n$ for some $\mu > 0$, see~\cite[Theorem~2]{Schwald20}.
We conclude by noting that the left-hand side is strictly positive.
Alternatively, one may also resort to the original definition of $q_X$ as a certain integral on a resolution of $X$.
There, one uses the Hodge--Riemann bilinear relations, cf.~the proof of~\cite[(4.5.1)]{GrafSchwald21}.
\end{proof}

\section{Characterization of torus quotients}

We are now in a position to prove \cref{thmA}.

\begin{setup} \label{setup main thm}
Let $X$ be a normal compact \kahler space of dimension $n$ with klt singularities and trivial first Chern class $\cc1X = 0 \in \HH2.X.\R.$.
Assume that $X$ is smooth in co\-dimension two.
\end{setup}

\begin{thm} \label{main thm}
In \cref{setup main thm}, assume that there exists a \kahler class $a \in \HH2.X.\R.$ such that $\cc2X \cdot a^{n-2} = 0$.
Then there exists a complex torus $T$ and a holomorphic action of a finite group $G \acts T$, free in codimension two, such that $X \isom \factor TG$.
\end{thm}

\begin{proof}
We proceed in three steps.

\emph{Step 1: Reduction to the split case.}
By~\cite[Theorem~A]{BakkerGuenanciaLehn20}, there exists a finite \qe cover $p \from \wt X \to X$ which decomposes as
\[ \wt X = T \x \prod_{i \in I} Y_i \x \prod_{j \in J} Z_j, \]
where $T$ is a complex torus, the $Y_i$ are ICY (= irreducible Calabi--Yau) varieties and the $Z_j$ are IHS (= irreducible holomorphic symplectic) varieties, cf.~\cref{defi IHS ICY}. In particular, $\wt X$ has canonical singularities because all its factors do.

Since $X$ is smooth in codimension two, the \qe map $p$ is necessarily \'etale in codimension two, and hence $\wt X$ is still smooth in codimension two.
Therefore
\[ \cc2{\wt X} \cdot (p^* a)^{n-2} = \deg(p) \cdot \cc2X \cdot a^{n-2} = 0 \]
by~\cite[Proposition~5.6]{GK20}.
Also, $p^* a$ is a \kahler class by~\cite[Proposition~3.6]{GK20}.
Finally, if the conclusion of \cref{main thm} holds for $\wt X$, then it also holds for $X$, by taking Galois closure~\cite[Lemma~2.8]{CGGN20}.
We may and will therefore replace $X$ by $\wt X$ (and $a$ by $p^* a$) for the remaining argument.
In order to finish the proof, it is sufficient to show that $I = J = \emptyset$ in the above notation. \\

\emph{Step 2: Chern class computations.}
The following calculation gets slightly messy due to the fact that we need to work on a resolution, but the basic idea is very simple. ---
Since $\HH1.Y_i.\R. = \HH1.Z_j.\R. = 0$ for all $i \in I, j \in J$, the K\"unneth formula implies that the class $a$ decomposes as
\begin{equation} \label{decomp a}
a = p_T^* \, a_T + \sum_{i \in I} p_i^* b_i + \sum_{j \in J} p_j^* c_j,
\end{equation}
where $a_T$ (resp.~$b_i$, $c_j$) is a \kahler class on the torus $T$ (resp.~on $Y_i$, $Z_j$) and the maps $p_\bullet$ are the projections.
We pick strong log resolutions $\widehat Y_i \to Y_i$, $\widehat Z_j \to Z_j$ and set $f \from \wX \to X$ where $\wX \defn T \x \prod_{i \in I} \widehat Y_i \x \prod_{j \in J} \widehat Z_j$, with projections $\wh p_\bullet$.
The first Chern class of each factor is (either zero or) supported on the exceptional divisor of $f$ and as $X$ is smooth in codimension two and $\cc2T = 0$, we have
\begin{equation} \label{decomp c2}
\cc2{\wh X} \cdot f^*(a^{n-2}) = \sum_{i \in I} \wh p_i^* \, \cc2{\wh Y_i} \cdot f^*(a^{n-2}) + \sum_{j \in J} \wh p_j^* \, \cc2{\wh Z_j} \cdot f^*(a^{n-2}).
\end{equation}
Here we have used \cref{637} to see that the ``mixed terms'' such as $\wh p_i^* \cc1{\wh Y_i} \cdot \wh p_j^* \cc1{\wh Z_j}$ vanish against $f^*(a^{n-2})$.
Putting together~\lref{decomp c2} and the pullback of~\lref{decomp a} to $\wh X$, and remembering our very definition of $\mathrm c_2$, as a result one gets
\begin{equation} \label{expand c2}
\cc2X \cdot a^{n-2} = \sum_{i \in I} \lambda_i \, \cc2{Y_i} \cdot b_i^{\dim Y_i - 2} + \sum_{j \in J} \mu_j \, \cc2{Z_j} \cdot c_j^{\dim Z_j - 2}.
\end{equation}
for some positive constants $\lambda_i, \mu_j > 0$.
More precisely, for $i_0 \in I$, $j_0 \in J$, we have up to some binomial coefficient $\lambda_{i_0} = a_T^{\dim T} \cdot \prod_{i \ne i_0} b_i^{\dim Y_i} \cdot \prod_{j \in J} c_j^{\dim Z_j}$ and similarly for $\mu_{j_0}$.

We have seen in \cref{positivity c2} that the numbers $\cc2{Y_i} \cdot b_i^{\dim Y_i - 2}$ and $\cc2{Z_j} \cdot c_j^{\dim Z_j - 2}$ are non-negative.
As the left-hand side of~\lref{expand c2} is zero, they must therefore all vanish. \\

\emph{Step 3: Eliminating the non-torus factors.}
By \cref{c2 IHS}, one must have $J = \emptyset$.
Assuming $i \in I \ne \emptyset$, note that the ICY variety $Y_i$ is projective by Kodaira's Embedding Theorem because $\HH2.Y_i.\O{Y_i}. = 0$.
The vanishing $\cc2{Y_i} \cdot b_i^{\dim Y_i - 2} = 0$ implies the vanishing of $\cc2{Y_i}$ against any ample class by \cref{change of class}.
Applying~\cite[Theorem~7.1]{GKP14} yields that $Y_i$ is a torus quotient, contradicting the definition of ICY varieties.
Hence $I = \emptyset$.
In particular, $X = T$ and the theorem is proved.
\end{proof}

\subsection*{Proof of \cref{thmA}}

The implication ``\lref{A.1} $\imp$ \lref{A.2}'' follows immediately from \cref{main thm}.
More precisely, as $X$ is smooth in codimension two, the resolution $f$ is clearly minimal in codimension two, being an isomorphism over $\Reg X$.
Therefore the assumptions of~\lref{A.1} imply that $\cc2X \cdot \alpha^{n-2} = 0$.

The other direction ``\lref{A.2} $\imp$ \lref{A.1}'' can be proven exactly as in \cite[Theorem~1.1]{GK20}:
let $\pi \from T \to \factor TG \isom X$ be the quotient map.
Since $\can X^{[\,|G|\,]} \isom \O X$, we have $\cc1X = 0 \in \HH2.X.\R.$.
By~\cite[Chapter~IV, Corollary~1.2]{Var}, $X$ is \kahler.
As $\pi$ is \'etale in codimension two, we have
\[ 0 = \cc2T \cdot \pi^* \alpha^{n-2} = \deg(\pi) \cdot \cc2X \cdot \alpha^{n-2} \]
for any \kahler class $\alpha$ on $X$, according to~\cite[Proposition~5.6]{GK20}.
\cref{thmA} is thus proved. \qed

\section{Open questions} \label{sec open}

In this last section, we outline a few natural open problems that fit in the framework of this paper.

\subsection{Actions free in codimension one}

First of all, one would like to study also finite group actions without the condition that they be free in codimension two.
If instead one imposes freeness in codimension one, then a close analog of \cref{thmA} is likely to hold.
However, one needs to replace the second Chern class from \cref{section chern classes} by the ``orbifold'' second Chern class $\ccorb2X$ introduced in~\cite[Definition~5.2]{GK20}.
In loc.~cit., the following conjecture has already been formulated and proven in dimension three.

\begin{conj}[\protect{=~\cite[Conjecture~1.3]{GK20}}] \label{c2orb conj}
Let $X$ be a compact complex space of dimension $n$ with klt singularities.
The following are equivalent:
\begin{enumerate}
\item We have $\cc1X = 0 \in \HH2.X.\R.$, and there exists a \kahler class $\alpha \in \HH2.X.\R.$ such that $\ccorb2X \cdot \alpha^{n-2} = 0$.
\item There exists a complex $n$-torus $T$ and a holomorphic action of a finite group $G \acts T$, free in codimension one, such that $X \isom \factor TG$.
\end{enumerate}
\end{conj}

\noindent
To attack \cref{c2orb conj} in general using the strategy of this paper, it is necessary to obtain a version of \cref{BG ineq} for orbifold Chern classes.

\subsection{General actions}

For general group actions, i.e.~without any freeness assumptions, one has to work with ``standard pairs''.
Below we recall this notion and we explain how the characterization of torus quotients in this setting can be reduced to \cref{c2orb conj} using a covering construction.

So let $(X, \Delta)$ be an $n$-dimensional compact klt pair, with $X$ a complex space and $\Delta$ a \emph{standard boundary}.
This means that $\Delta = \sum_{i \in I} \big( 1 - \frac1{m_i} \big) \Delta_i$ where $m_i \geq 2$ are integers.
We can define the \emph{orbifold second Chern class} $\ccorb2{X, \Delta}$ of the pair $(X, \Delta)$ in the spirit of~\cite[\S 5]{GK20}:
let $U \subset X$ be the open subset where the pair $(U, \Delta|_U)$ is an orbifold.
By that, we mean that one can cover $U$ by euclidean open subsets $U_\alpha$ admitting a finite, surjective Galois cover $p_\alpha \from V_\alpha \to U_\alpha$ from a smooth manifold $V_\alpha$ such that $p_\alpha$ ramifies in codimension one exactly along $\supp(\Delta) \cap U_\alpha$, of order $m_i$ along $\Delta_i \cap U_\alpha$.
Equivalently, one has $K_{V_\alpha} = p_\alpha^* \big( K_{U_\alpha} + \Delta|_{U_\alpha} \big)$.
For such an orbifold pair $(U, \Delta|_U)$, we consider the locally $V$-free sheaf on $U$ given by $T_{V_\alpha}$ in each chart $p_\alpha$.
It is standard to associate Chern classes to a locally $V$-free sheaf, cf.~e.g.~\cite{Satake56} or~\cite[(1.6)--(1.9) \& (2.10)]{Blache96}.
In particular, we obtain an element $\mathrm c_2^{\mathrm{orb}}(U, \Delta|_U) \in \HH4.U.\R.$.

Note that $X \setminus U$ has codimension at least three in $X$ and hence the natural map $\HHc2n-4.U.\R. \to \HH2n-4.X.\R.$ is an isomorphism.
We define the class
\[ \ccorb2{X, \Delta} \in \HH2n-4.X.\R.\dual \]
as the unique element whose restriction to $U$ is the class $\mathrm c_2^{\mathrm{orb}}(U, \Delta|_U) \in \HH4.U.\R. = \HHc2n-4.U.\R.\dual$, where the last equality is Poincar\'e duality.

With all preliminaries in place, we suggest the following analog of \cref{c2orb conj}:

\begin{conj}[Characterization of torus quotient pairs] \label{c2 pair conj}
In the above setting, the following are equivalent:
\begin{enumerate}
\item\label{138} We have $\cc1{K_X + \Delta} = 0 \in \HH2.X.\R.$, and there exists a \kahler class $\alpha \in \HH2.X.\R.$ such that $\ccorb2{X, \Delta} \cdot \alpha^{n-2} = 0$.
\item\label{139} There exists a complex $n$-torus $T$ and a holomorphic action of a finite group $G \acts T$ such that $X \isom \factor TG$ and $\supp \Delta$ is exactly the codimension one part of the branch locus of the quotient map $T \to X$.
More precisely: for each $i \in I$, the map $T \to X$ is branched over $\Delta_i$ with multiplicity $m_i$, and it is \qe when restricted to $X \setminus \supp \Delta$.
\end{enumerate}
\end{conj}

Evidence for \cref{c2 pair conj} is provided by the following result.

\begin{prop} \label{conj implication}
\cref{c2orb conj} implies \cref{c2 pair conj}.
\end{prop}

\begin{proof}
Assume~\lref{138}.
According to~\cite[Corollary~1.18]{JHM2}, Abundance holds for the pair $(X, \Delta)$, so $K_X + \Delta$ is \Q-linearly equivalent to zero.
As explained\footnote{We would like to thank St\'ephane Druel for bringing Shokurov's construction to our attention.} in~\cite[Example~(2.4.1)]{Shokurov92} we can find a finite cyclic cover $\pi \from \wt X \to X$ that branches exactly over $\Delta_i$ with multiplicity~$m_i$.
In other words, we have
\[ K_{\wt X} = \pi^*(K_X + \Delta), \]
and $\wt X$ has canonical singularities and vanishing first Chern class.
Set $\wt U \defn \pi\inv(U)$; we claim that each point $y \in \wt U$ has a neighborhood $\wt U_y$ admitting a \qe cover $\wt V_y \to \wt U_y$ where $\wt V_y$ is smooth.
In particular $\wt U$ is included in the orbifold locus of $\wt X$ and $\wt X \setminus \wt U$ has codimension at least three.

In order to check the claim, let $x = \pi(y)$ and let $U_x$ be a neighborhood of $x$ admitting a finite cover $p_x \from  V_x \to U_x$ where $V_x$ is smooth, $p_x$ ramifies at order $m_i$ along $\Delta_i \cap U_x$ and nowhere else in codimension one.
Set $\wt U_x = \pi\inv(U_x)$ and let $\wt V_x$ be the normalization of $\wt U_x \x_{U_x} V_x$; we have the following diagram
\[ \xymatrix{
V_x \ar_-{p_x}[d] & & \wt V_x \ar[d] \ar[ll] \\
U_x & & \wt U_x \ar^-\pi[ll]
} \]
Since $p_x$ and $\pi$ ramify at the same order in codimension one, $\wt V_x \to V_x$ and $\wt V_x \to \wt U_x$ are \qe.
Since $V_x$ is smooth, it implies that $\wt V_x \to V_x$ is \'etale.
In particular, $\wt V_x$ is smooth as well, hence the claim.

By the above, if $h$ is an orbifold hermitian metric on $T_U$, then $\pi^* h$ is an hermitian metric on $T_{\wt U}$ and $\pi^* \mathrm c_2^{\mathrm{orb}} \big( (U, \Delta|_U), h \big) = \cc2{\wt U, \pi^* h}$ as (orbifold) forms of degree $4$.
Recalling that $\alpha \in \HH2.X.\R.$ is the given \kahler class, let $a \in \alpha^{n-2}$ be a representative with compact support in $U$ (that is, it is an orbifold form of degree $2n - 4$).
By definition, one has 
\begin{align*}
\ccorb2{X, \Delta} \cdot \alpha^{n-2} & = \int_U \mathrm c_2^{\mathrm{orb}} \big( (U, \Delta|_U), h \big) \wedge a \\
& = \frac{1}{\deg \pi} \int_{\wt U} \cc2{\wt U, \pi^* h} \wedge \pi^*a \\
& = \frac{1}{\deg \pi} \ccorb2{\wt X} \cdot \pi^* \alpha^{n-2},
\end{align*}
hence $\ccorb2{\wt X} \cdot \pi^* \alpha^{n-2} = 0$.
As $\pi^* \alpha$ is still a \kahler class, by \cref{c2orb conj} there is a \qe Galois map $T' \to \wt X$ with $T'$ a complex torus.
Considering the Galois closure of the composition $T' \to \wt X \to X$ proves~\lref{139}.
The other implication is similar, but easier (in particular it does not rely on \cref{c2orb conj}).
\end{proof}

\bibliographystyle{alpha}
\bibliography{biblio}

\end{document}